\documentclass{amsart}

\usepackage{amsmath,amsthm,amssymb}
\usepackage{graphicx,color,tikz,wrapfig,subcaption}
\usepackage{comment}

\newtheorem{thm}{Theorem}[section]

\newtheorem{cor}[thm]{Corollary}
\newtheorem{conj}[thm]{Conjecture}
\newtheorem{lemma}[thm]{Lemma}

\theoremstyle{definition}
\newtheorem{ex}[thm]{Example}
\newcommand{\dist}{\text{dist}}

\begin{document}

\title{Randi\'c index, radius, and diameter for cactus graphs}

\author[M. Doig]{Margaret I. Doig}
\address{Department of Mathematics, Creighton University}
\email{margaretdoig@creighton.edu}
\date{\today}

\thanks{Supported by CURAS Summer Faculty Research Fund}

\begin{abstract}
We study the Randi\'c index for cactus graphs. It is conjectured to be bounded below by radius (for other than an even path), and it is known to obey several bounds based on diameter. We study radius and diameter for cacti then verify the radius bound and strengthen two diameter bounds for cacti. Along the way, we produce several other bounds for the Randi\'c index in terms of graph size, order, and valency for several special classes of graphs, including chemical nontrivial cacti and cacti with starlike BC-trees.
\end{abstract}

\maketitle

\section{Introduction}

\subsection*{Scientific motivation} 

A number of molecular properties appear to depend on the shape of a molecule and vary even between different arrangements of the same atoms; for example, the boiling points of hexane isomers appear to be correlated to the surface area of the molecular cloud, which reflects the degree and location of branching of the molecule. In 1975, Milan Randi\'c proposed a molecular branching index based entirely on molecular graphs in an attempt to mathematically characterize branching in a way consistent with boiling point and other structure-related properties such as the enthalpy of formation of alkanes and the relationship of vapor pressure to temperature. Besides simplifying study of such properties and allowing their prediction for novel molecules, such a mathematical characterization also reveals a correlation between preexisting constants such as some of the Antoine coefficients and therefore allows a reduction in the number of constants which must be experimentally determined \cite{randic1975characterization}. Since then, Randi\'c's index has become a standard tool for evaluating molecular structure in quantitative structure-activity relationship (QSAR) models, that is, regressive models that predict biological activity, physicochemical properties, and toxicological responses of chemical compounds based on their molecular structure (see, for example, \cite{kier1986molecular, pogliani2000molecular, garcia2008some, todeschini2008handbook, kier2012molecular}).

\subsection*{Mathematical investigations}
As it depends exclusively (at least as phrased initially) on the arrangement of atoms within a molecule, the Randi\'c index is a graph theoretic invariant. Randi\'c's original formulation was based on the graph adjacency matrix, and the common formula is due to Balaban in 1982 \cite{balaban1982discriminating}. The \emph{Randi\'c index} of a graph $G$ is
\[R=R(G) = \sum_{e} w(e)\]
where the sum runs over all edges $e$ and $w(e)$ is a weight assigned to each edge, that is, if $e=uv$ is an edge in $E$ and $d_u$ is the degree of a vertex $u$, then
\[w(e) = \frac{1}{\sqrt {d_u d_v}}.\]
This invariant was recast by Caporossi, Gurman, Hansen, and Pavlovic in 2003 \cite{caporossi2003graphs} as 
\[R = \frac{n-n_0}{2} - \sum_{e\in E} w^\ast(e)\]
where $n$ is the number of vertices and $n_0$ the number of isolated vertices, and $w^\ast(e)$ is a measure of the asymmetry of edge weights, that is, 
\begin{equation}\label{eqn:weight}w^\ast(e) = \frac{1}{2}\left(\frac{1}{\sqrt{d_u}}-\frac{1}{\sqrt{d_v}}\right)^2.\end{equation}
An immediate result, as indicated by the authors, is an upper bound on the Randi\'c index:
\[R \leq \frac{n}{2}\]
and, in fact, this bound is attained by regular graphs $K_n$, where every vertex in each component has the same degree and therefore every edge has $w^\ast = 0$.

Due to its success modeling physicochemical properties of molecules, there is a great deal of interest in approximating the Randi\'c index and understanding how it changes under certain structural alterations, or bounding it on particular classes of graphs. It has been compared to the minimum/maximum degrees of a graph~\cite{bollobas1998extremal, suil2018sharp, divnic2013proof, liu2013conjecture}, chromatic index~\cite{li2010relation}, average path length~\cite{caporossi2000variable}, graph eigenvalues~\cite{araujo1998connectivity, aouchiche2006variable, elphick2015bounds}, graph matching~\cite{aouchiche2006variable}, and radius and diameter. Our efforts focus on the last two, which we will define below.

Radius was first explicitly connected to the Randi\'c index in the ground-breaking work in the late 1980s by Fajtlowicz, who used a computer search of a large database of graphs and invariants to conjecture possible relationships between invariants. He proposed the radius as a possible lower bound for $R$: 
\begin{conj}\cite{fajtlowicz1988graffiti}\label{conj:r}
Let $G$ be a graph. If it is an even path, then:
\[R -r \geq \sqrt 2 - \frac{3}{2}\]
and, otherwise,
\[R -r \geq 0.\]
\end{conj}
\noindent Note the original conjecture was that $R \geq r-1$, which was later modified to $R \geq r$ for graphs other than even paths, where $R \approx r-0.1$.

Caporossi and Hansen verified $R-r \geq \sqrt 2 - \frac{3}{2}$ for trees in 2000~\cite{caporossi2000variable}; Cygan, Pilipczuk, and \v{S}krekovski $R-r \geq -\frac{1}{2}$ for chemical graphs in 2012 \cite{cygan2012inequality}; Liu and Gutman $R-r \geq -1$ for several special classes of graphs, including unicyclic and bicyclic graphs, in 2009~\cite{liu2009conjecture}; You and Liu $R-r \geq -1$ for tricyclic graphs and $R-r \geq 0$ for biregular graphs and graphs up to 10 vertices in 2009~\cite{you2009conjecture}. We complete the proof of the sharper bound $R-r \geq 0$ in Theorems~\ref{thm:ntc_r} and \ref{thm:cactus_r} for cacti (graphs in which no two cycles share an edge):
\begin{cor}\label{cor:r}
Let $G$ be a cactus. If $G$ is not an even path,
\[R-r \geq 0.\]
In fact, if $G$ has $k>0$ cycles,
\[R-r \geq (k-1)\left(\sqrt 2 - 1\right)\]
and, if $G$ has $m>0$ bridges,
\[R-r \geq (k-1)\left(\sqrt2-1\right)+m\left(\frac{1}{\sqrt 3} + \sqrt\frac{2}{3}-1\right) - \frac{1}{2}.\]
\end{cor}

Aouchiche later introduced diameter into the discussion by proposing a pair of bounds on $R$ in terms of diameter in 2007. 
\begin{conj}\cite{aouchiche2007conjecture}
Let $G$ be a graph with $n$ vertices. Then
\[\begin{split}
R - d &\geq \sqrt 2 - \frac{n+1}{2}\\
\frac{R}{d} &\geq \frac{n-3+2\sqrt 2}{2n-2}.
\end{split}\]
\end{conj}

Liu and Zhang verified these bounds for unicyclic graphs in 2010~\cite{zhang2010conjecture}, and Li and Shi verified $R-d$ for graphs with smallest degree at least 5 and $\frac{R}{d}$ for all graphs where the total number of vertices is not too much larger than the smallest degree~\cite{li2009randi}. Yang and Lu verified both bounds for all graphs and introduced another bound on $R-\frac{d}{2}$ that they proved for trees~\cite{yang2011randic}, and, simultaneously, J. Liu, Liang, Cheng, and B. Liu verified the bounds on $R-d$ for all graphs and $\frac{R}{d}$ for trees~\cite{liu2011proof}.
\begin{thm}\label{thm:d} 
Let $G$ be a graph on $n$ vertices and $e$ edges. Then: 
\begin{itemize}
\item \cite{yang2011randic,liu2011proof} \[R - d \geq - \frac{e}{2} + \sqrt 2 - 1 ,\]
\item \cite{yang2011randic} \[\frac{R}{d} \geq \frac{n-3+2\sqrt 2}{n+e-1},\]
\item \cite{yang2011randic} and, if $G$ is a tree, \[R-\frac{d}{2} \geq \sqrt 2 - 1,\]
\end{itemize}
all with equality iff $G$ is a path.
\end{thm}

We propose sharper versions of these bounds which also consider the number of cycles and number of bridges. We verify stronger bounds on $R-d$ and $R-\frac{d}{2}$ for cacti, and we conjectures a stronger bound on $\frac{R}{d}$ which we demonstrate for the special case where the BC-tree is starlike.

\begin{cor}\label{cor:d}
Let $G$ be a cactus with $k$ cycles and $b$ bridges. If $G$ is a nontrivial cactus ($b = 0$), 
\[\begin{split}
R-d &\geq -(k-1)\left(2-\sqrt 2\right)\\
R-\frac{d}{2} &\geq \frac{n}{4}-(k-1)\left(\frac{7}{4}-\sqrt 2\right)
\end{split}\]
with equality if the graph has BC-tree a path and is longitudinally symmetric. If $k>0$ and $e>0$,
\[\begin{split}
R-d &\geq -\frac{b}{2}-(k-1)\left(2-\sqrt 2\right) - 3 + \frac{2}{\sqrt 3} + 2 \sqrt{\frac{2}{3}}\\
R-\frac{d}{2} &\geq \frac{n-b}{4}-(k-1)\left(\frac{7}{4}-\sqrt 2\right)- 3 + \frac{2}{\sqrt 3} + 2 \sqrt{\frac{2}{3}}
\end{split}\]
with equality if the graph has BC-tree a path, has two leaves, and is longitudinally symmetric. 
\end{cor}

\begin{conj}\label{conj:d_new}
Let $G$ be a cactus with $k$ cycles and $b$ bridges. For a nontrivial cactus ($b = 0$), 
\[\frac{R}{d} \geq \frac{n - (k-1)(3-2\sqrt 2)}{n+k-1}\]
with equality if the graph has BC-tree a path and is longitudinally symmetric. For a cactus with $k>0$ and $b>0$,
\[\frac{R}{d} \geq \frac{ n - (k-1)\left(3 - 2\sqrt 2 \right) - 6 + \frac{4}{\sqrt 3} + 4\sqrt\frac{2}{3}}{n+k+b-1}\]
with equality if the graph has BC-tree a path, has two leaves, and is longitudinally symmetric. 
\end{conj}

In the process of proving our bounds, we also develop several bounds on the Randi\'c index of cacti in terms of valency. This falls into the established literature of bounds on $R$ in terms of graph order or size (see \cite{favaron2003randic} for a survey). There are some bounds like those of Caporossi, Gutman, Hansen, and Pavlovi\'c that, fixing $n$, $R$ attains its maximum on complete graphs~\cite{caporossi2003graphs}, or of Yang and Lu that it attains its maximum among trees on paths of length at least 2~\cite{yang2011randic}. Similarly, Bollob\'as and Erd\"os showed that $R$ reaches its minimum on the star~\cite{bollobas1998extremal}, and Lu, Zhang, and Tian that it reaches its minimum among cacti on a bouquet of triangles and pendants~\cite{lu2006randic}. There is also a thriving literature of bounds in terms of maximum or minimum valency (including \cite{gutman2000graphs, aouchiche2007variable, suil2018sharp, suil2018sharp}), but we are unaware of any which consider a sum of valencies. We prove:
\begin{cor}\label{cor:valency}
Let $G$ be a cactus on $n$ vertices with $k$ cycles and $b$ bridges where $d_v$ is the degree of a vertex $v$. If $G$ is a tree ($k=0$), 
\[R \geq 1 - n + \sum_v \sqrt{d_v}\]
with equality for a star. Otherwise,
\[R \geq \frac{1 - n-k}{2} - b\left(\frac{3}{2}-\sqrt2\right) + \sum_v \sqrt{\frac{d_v}{2}} \]
with equality for a nontrivial cactus ($b=0$) none of whose articulation points are adjacent.
\end{cor}


\subsection*{Definitions and notation.} Let $G=(V,E)$ be a graph with vertex set $V$ and edge set $E$. The degree of a vertex $v$ will be denoted $d_v$, and the size of a cycle $c$ will be $s_c$. \emph{Graph order} is the number of vertices, and \emph{graph size} is the number of edges. A \emph{cut vertex} or \emph{articulation point} is a vertex whose removal increases the number of connected components of a graph, and a \emph{cut edge} or \emph{bridge} is an edge whose removal does the same. A \emph{leaf} is a vertex of degree 1, and a \emph{pendant} is an edge with a leaf at one end. If required for intelligibility, we will write $R(G)$ for $R$, $d_v(G)$ for $d_v$, and so on.

Common graphs we use are: the \emph{path} on $n$ vertices, $P_n$; the \emph{cycle} on $n$ vertices, $C_n$; the \emph{star} on $n+1$ vertices $S_n$. A graph is \emph{starlike} with $r>2$ \emph{arms} if it has one vertex of degree $r$ called its \emph{root} to which are attached $r$ paths. We may refer to a path or cycle as \emph{even} (respectively, \emph{odd}) when it has an even (respectively, \emph{odd}) number of vertices. A \emph{tree} is a connected graph where every vertex is either a leaf or an articulation point; equivalently, every edge is a bridge. It is the connected graph with the fewest edges for a given number of vertices. A \emph{cactus} is a connected graph where any two cycles intersect in at most one point, and a \emph{nontrivial cactus} is a cactus without bridges; in other words, a nontrivial cactus may be constructed by iteratively adding cycles to one another (glued together by identifying a vertex from each), and a cactus by adding cycles and bridges. 

We may sometimes divide more complicated graphs up into smaller subgraphs for convenience. A \emph{block} is a maximal biconnected subgraph, that is, a subgraph without any articulation points; equivalently, the articulation points of a graph divide it into connected subgraphs which we call blocks. We may draw a graph of these blocks as in Figure~\ref{fig:bc-tree}, called the \emph{block cut tree} or {BC-tree}: assign a vertex in the BC-tree for each articulation point and for each block in the graph, and assign an edge in the BC-tree for each pair of an articulation point and a block which contains it. Note that it is indeed a tree.

\begin{figure}
\begin{center}
\begin{subfigure}[scale=0.5]{0.5\textwidth}
\begin{center}

\begin{tikzpicture}
\draw[red, thick] (0,0) -- (0.8, 0);
\draw[red, thick] (0,0) -- (0.4, 0.7);
\draw[red, thick] (0.8,0) -- (0.4, 0.7);

\draw[yellow, thick] (0.8,0) -- (0.8, -0.8);

\draw[green, thick] (0.8,0) -- (1.6, 0);

\draw[blue, thick] (2.4,0) -- (1.6, 0);
\draw[blue, thick] (1.6,0.8) -- (1.6, 0);
\draw[blue, thick] (2.4,0.8) -- (2.4, 0);
\draw[blue, thick] (2.4,0.8) -- (1.6, 0.8);
\draw[blue, thick] (2.4,0.8) -- (1.6, 0);
\draw[blue, thick] (2.4,0) -- (2.05, 0.35);
\draw[blue, thick] (1.95,0.45) -- (1.6, 0.8);
\draw[blue, thick] (2.4,0.8) -- (3.1, 0.4);
\draw[blue, thick] (2.4,0) -- (3.1, 0.4);

\draw[orange, thick] (3.1,0.4) -- (3.9, 0.4);

\draw[purple, thick] (2.4,0) -- (2.8, -0.7);

\filldraw[red] (0,0) circle (2pt);
\filldraw[white] (0.8,0) circle (2pt);
\draw[black] (0.8,0) circle (2pt)node[anchor=north east] {$v_1$};
\filldraw[red] (0.4,0.7) circle (2pt);
\filldraw[yellow] (0.8,-0.8) circle (2pt);
\filldraw[white] (1.6,0) circle (2pt);
\draw[black] (1.6,0) circle (2pt)node[anchor=north] {$v_2$};
\filldraw[blue] (1.6,0.8) circle (2pt);
\filldraw[white] (2.4,0) circle (2pt);
\draw[black] (2.4,0) circle (2pt)node[anchor=north west] {$v_3$};
\filldraw[blue] (2.4,0.8) circle (2pt);
\filldraw[white] (3.1,0.4) circle (2pt);
\draw[black] (3.1,0.4) circle (2pt)node[anchor=south] {$v_4$};
\filldraw[purple] (2.8,-0.7) circle (2pt);
\filldraw[orange] (3.9,0.4) circle (2pt);
\end{tikzpicture}
\end{center} 
\subcaption{}\label{fig:bc-tree-a}
\end{subfigure}
\begin{subfigure}[scale=0.5]{0.4\textwidth}
\begin{center}
\begin{tikzpicture}
\draw[black, thick] (-0.9,0.4) -- (-0.2, 0);

\draw[black, thick] (-0.2,0) -- (-0.2, -0.8);

\draw[black, thick] (-0.2,0) -- (0.6, 0);
\draw[black, thick] (1.4,0) -- (0.6, 0);

\draw[black, thick] (1.4,0) -- (1.8,0.7);

\draw[black, thick] (1.8,0.7) -- (2.2,0);
\draw[black, thick] (2.2,0) -- (2.6, -0.7);

\draw[black, thick] (1.8,0.7) -- (2.6, 0.7);
\draw[black, thick] (2.6,0.7) -- (3.4, 0.7);

\filldraw[white] (-0.2,0) circle (2pt);
\draw[black] (-0.2,0) circle (2pt)node[anchor=north east] {$v_1$};
\filldraw[red] (-0.9,0.4) circle (4pt);
\filldraw[yellow] (-0.2,-0.8) circle (4pt);
\filldraw[green] (0.6,0) circle (4pt);
\filldraw[white] (1.4,0) circle (2pt);
\draw[black] (1.4,0) circle (2pt)node[anchor=north] {$v_2$};
\filldraw[blue] (1.8,0.7) circle (4pt);
\filldraw[white] (2.2,0) circle (2pt);
\draw[black] (2.2,0) circle (2pt)node[anchor=west] {$v_3$};
\filldraw[purple] (2.6, -0.7) circle (4pt);
\filldraw[white] (2.6,0.7) circle (2pt);
\draw[black] (2.6,0.7) circle (2pt)node[anchor=south] {$v_4$};
\filldraw[orange] (3.4, 0.7) circle (4pt);
\end{tikzpicture} 
\end{center} 
\subcaption{}\label{fig:bc-tree-b}
\end{subfigure}
\end{center}
\caption{(A) A graph colored by block; articulation points belong to all incident blocks and are left uncolored. This graph has one triangle, four bridges (including three pendants), and one more complex block. (B) The corresponding BC-tree. Colored vertices correspond to blocks, hollow vertices to articulation points, and edges to articulation point/block pairs.}\label{fig:bc-tree}
\end{figure}
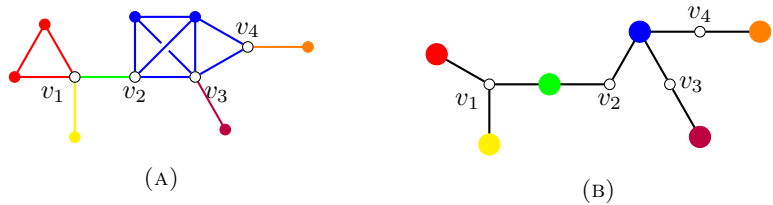

We study these graphs by use of two standard invariants, radius and diameter, and a degree-based topological index, the Randi\'c index. The \emph{distance} $\dist(u,v)$ between two vertices is the minimal length of a path between them, and the \emph{eccentricity} $\text{ecc}(v)$ of a vertex is the maximum distance from that vertex to any other vertex in its component,
\[\text{ecc}(v)=\max_{u\in V} \dist(u,v).\]
The \emph{radius} is the minimum eccentricity
\[r = \min_{v\in V} \text{ecc}(v),\]
while the diameter is the maximum distance
\[d = \max_{u,v\in V} \dist(u,v).\]

\subsection*{Organization} We will be building graphs inductively, so, in Section~\ref{sec:induction}, we introduce sample calculations to develop the reader's intuition and establish facts about some base graphs, and we also discuss how $R$ changes under addition of a pendant edge or a cycle to a graph. We will also establish in Section~\ref{sec:rd_lemmata} some facts about the radius and diameter of a cactus which we will need when we perform these block additions. Next, in Section~\ref{sec:osv}, we bound $R$ on cacti using vertex valency as well as the number of cycles and bridges. Finally, in Section~\ref{sec:rd}, we verify the bounds on $R$ in terms of $r$ and $d$ from Corollaries~\ref{cor:r} and \ref{cor:d}.

\subsection*{Acknowledgements}
Thanks to CURAS at Creighton University for summer research support. Thanks also to Anna Rossini for showing up in my office wanting to ``do something'' in mathematical chemistry, which started me on this interesting trip away from my normal research areas.

\section{Block addition and its effects on $R$}\label{sec:induction}

\subsection{Examples}

We provide a few examples for the reader's convenience. Observe that the conjectures above are all satisfied.

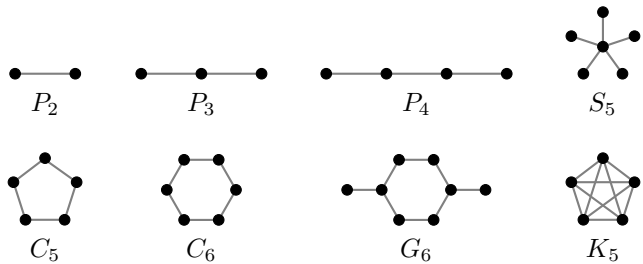
\begin{center}
\begin{tabular}{cccccccccccccc}
\begin{tikzpicture}
\draw[gray, thick] (0,0) -- (0.8, 0);
\filldraw[black] (0,0) circle (2pt);
\filldraw[black] (0.8,0) circle (2pt);
\end{tikzpicture} 
&
\qquad
&
\begin{tikzpicture}
\draw[gray, thick] (0,0) -- (1.6, 0);
\filldraw[black] (0,0) circle (2pt);
\filldraw[black] (0.8,0) circle (2pt);
\filldraw[black] (1.6,0) circle (2pt);
\end{tikzpicture} 
&
\qquad 
&
\begin{tikzpicture}
\draw[gray, thick] (0,0) -- (2.4, 0);
\filldraw[black] (0,0) circle (2pt);
\filldraw[black] (0.8,0) circle (2pt);
\filldraw[black] (1.6,0) circle (2pt);
\filldraw[black] (2.4,0) circle (2pt);
\end{tikzpicture} 
&
\qquad
&
\begin{tikzpicture}
\draw[gray, thick] (0,0) -- (0.42,0.14);
\draw[gray, thick] (0,0) -- (0.26,-0.36);
\draw[gray, thick] (0,0) -- (-0.26,-0.36);
\draw[gray, thick] (0,0) -- (-0.42,0.14);
\draw[gray, thick] (0,0) -- (0,0.46);
\filldraw[black](0,0) circle (2pt);
\filldraw[black] (0,0.46) circle (2pt);
\filldraw[black] (0.42,0.14) circle (2pt);
\filldraw[black] (0.26,-0.36) circle (2pt);
\filldraw[black] (-0.26,-0.36) circle (2pt);
\filldraw[black] (-0.42,0.14) circle (2pt);
\end{tikzpicture} \\
$P_2$ && $P_3$ && $P_4$ && $S_5$\\ \\
\begin{tikzpicture}
\draw[gray, thick] (0,0.44) -- (0.42,0.14);
\draw[gray, thick] (0.42,0.14) -- (0.26,-0.36);
\draw[gray, thick] (0.26,-0.36) -- (-0.26,-0.36);
\draw[gray, thick] (-0.26,-0.36) -- (-0.42,0.14);
\draw[gray, thick] (-0.42,0.14) -- (0,0.46);
\filldraw[black] (0,0.46) circle (2pt);
\filldraw[black] (0.42,0.14) circle (2pt);
\filldraw[black] (0.26,-0.36) circle (2pt);
\filldraw[black] (-0.26,-0.36) circle (2pt);
\filldraw[black] (-0.42,0.14) circle (2pt);
\end{tikzpicture} 
&
\qquad
&
\begin{tikzpicture}
\draw[gray, thick] (0.46,0) -- (0.23,0.4);
\draw[gray, thick] (0.23,0.4) -- (-0.23,0.4);
\draw[gray, thick] (-0.23,0.4) -- (-0.46,0);
\draw[gray, thick] (-0.46,0) -- (-0.23,-0.4);
\draw[gray, thick] (-0.23,-0.4) -- (0.23,-0.4);
\draw[gray, thick] (0.23,-0.4) -- (0.46,0);
\filldraw[black] (0.46,0) circle (2pt);
\filldraw[black] (0.23,0.4) circle (2pt);
\filldraw[black] (-0.23,0.4) circle (2pt);
\filldraw[black] (-0.46,0) circle (2pt);
\filldraw[black] (-0.23,-0.4) circle (2pt);
\filldraw[black] (0.23,-0.4) circle (2pt);
\end{tikzpicture} 
&
\qquad
&
\begin{tikzpicture}
\draw[gray, thick] (-0.46,0) -- (-0.92,0);
\draw[gray, thick] (0.46,0) -- (0.23,0.4);
\draw[gray, thick] (0.23,0.4) -- (-0.23,0.4);
\draw[gray, thick] (-0.23,0.4) -- (-0.46,0);
\draw[gray, thick] (-0.46,0) -- (-0.23,-0.4);
\draw[gray, thick] (-0.23,-0.4) -- (0.23,-0.4);
\draw[gray, thick] (0.23,-0.4) -- (0.46,0);
\draw[gray, thick] (0.46,0) -- (0.92,0);
\filldraw[black] (0.92,0) circle (2pt);
\filldraw[black] (0.46,0) circle (2pt);
\filldraw[black] (0.23,0.4) circle (2pt);
\filldraw[black] (-0.23,0.4) circle (2pt);
\filldraw[black] (-0.46,0) circle (2pt);
\filldraw[black] (-0.92,0) circle (2pt);
\filldraw[black] (-0.23,-0.4) circle (2pt);
\filldraw[black] (0.23,-0.4) circle (2pt);
\end{tikzpicture} 
&
\qquad
&
\begin{tikzpicture}
\draw[gray, thick] (0,0.46) -- (0.42,0.14);
\draw[gray, thick] (0,0.46) -- (0.26,-0.36);
\draw[gray, thick] (0,0.46) -- (-0.26,-0.36);
\draw[gray, thick] (0,0.46) -- (-0.42,0.14);
\draw[gray, thick] (0.26,-0.36) -- (0.42,0.14);
\draw[gray, thick] (0.26,-0.36) -- (-0.42,0.14);
\draw[gray, thick] (0.26,-0.36) -- (-0.26,-0.36);
\draw[gray, thick] (-0.26,-0.36) -- (0.42,0.14);
\draw[gray, thick] (-0.26,-0.36) -- (-0.42,0.14);
\draw[gray, thick] (-0.42,0.14) -- (0.42,0.14);
\filldraw[black] (0,0.46) circle (2pt);
\filldraw[black] (0.42,0.14) circle (2pt);
\filldraw[black] (0.26,-0.36) circle (2pt);
\filldraw[black] (-0.26,-0.36) circle (2pt);
\filldraw[black] (-0.42,0.14) circle (2pt);
\end{tikzpicture} 
\\

$C_5$ && $C_6$ && $G_6$ && $K_5$
\end{tabular}
\end{center}

\begin{ex}\label{ex:path}
For $P_2$, 
\[d=r=R=1;\]
for all longer paths $P_n$,
\[d=n-1, \qquad r = \left\lfloor \frac{n}{2} \right\rfloor, \qquad R=\frac{n-3}{2} + \sqrt 2.\]
\end{ex}

\begin{ex}\label{ex:star}
The star $S_1$ is just the path $P_2$, and all bigger stars $S_{n-1}$ satisfy: 
\[d=2, \qquad r=1, \qquad R=\sqrt {n-1}.\]
\end{ex}

\begin{ex}\label{ex:cycle}
The cycle like $C_n$ obeys: 
\[d = r = \left\lfloor \frac{n}{2} \right\rfloor, \qquad R = \frac{n}{2}.\]
\end{ex}

\begin{ex}\label{ex:lollipop}
For a graph like $G_6$ above with one cycle and two non-adjacent pendants (so, in particular, the cycle has at least 4 vertices),
\[d=2+\left\lfloor\frac{s}{2}\right\rfloor, \qquad r=\left\lfloor\frac{s}{2}\right\rfloor, \qquad R = \frac{2}{\sqrt{3}} + 2\sqrt\frac{2}{3} + \frac{s-4}{2}\]
\end{ex}

\begin{ex}\label{ex:complete}
For a complete graph $K_n$,
\[d = r = 1, \qquad R = \frac{n}{2}.\]
\end{ex}

\subsection{Changing $R$}

When we perform explicit calculations of $R$, we will construct a graph inductively by starting with a path or a cycle and adding pendants or cycles one at a time carefully. We will need a few technical lemmas about the effect on $R$ of adding a pendant or a cycle.

\begin{lemma}\label{lemma:leaf}
Let $G$ be a graph on $n$ vertices. Add a pendant $e$ at vertex $v$ to form $G+e$. Then:
\begin{enumerate}
\item\cite[Lemma~1]{bollobas1998extremal}
\[R(G+uv)-R(G) \geq \sqrt{d_v+1}-\sqrt{d_v}\]
with equality iff all vertices adjacent to $v$ are leaves. 
\item If $d_w \geq 2$ for all vertices $w$ adjacent to $v$, then
\[R(G+uv)-R(G) > \sqrt{d_v+1} - \sqrt{\frac{d_v}{2}}\]
with equality iff all vertices adjacent are degree 2.
\end{enumerate}
\end{lemma}

\begin{proof} 
We present the proof of the first part with slight modification from the original; we are unaware of the presence of the second part in the literature.

If we add a pendant $l$ at $v$, then the only edges whose weights are affected by the transformation are in $N(v)$, the edges incident to $v$. We define:
\[S = \sum_{w\in N(v)}\frac{1}{\sqrt{d_w}}.\]
Thus, adding a single pendant at $v$ alters $R$ by:
\[R(G+l)-R(G) = \frac{1}{\sqrt{d_v+1}}-S\left(\frac{1}{\sqrt{d_v}}-\frac{1}{\sqrt{d_v+1}}\right).\]
If we know that $\delta \leq d_w \leq D$ for all $w \in N(v)$, then
\[\frac{d_v}{\sqrt{D}} \leq S \leq \frac{d_v}{\sqrt{\delta}}.\]

Thus, if all $d_w \geq 1$, we know that $S \leq d_v$, and so also 
\[R(G+l)-R(G) \geq \frac{1}{\sqrt{d_v+1}}-d_v \left(\frac{1}{\sqrt{d_v}}-\frac{1}{\sqrt{d_v+1}}\right) = \sqrt{d_v+1}-\sqrt{d_v},\]
and this bound is sharp if $d_w=1$ for all $w \in N(v)$. Similarly, if all $d_w \geq 2$, then $S \leq \frac{d_v}{\sqrt 2}$, and
\[R(G+l)-R(G) \geq \frac{1}{\sqrt{d_v+1}}-\frac{d_v}{\sqrt 2} \left(\frac{1}{\sqrt{d_v}}-\frac{1}{\sqrt{d_v+1}}\right)
> \frac{\sqrt{d_v+1}-\sqrt{d_v}}{\sqrt 2}.\]
\end{proof}

\begin{lemma} \label{lemma:cycle}
Let $G$ be a graph. Add a cycle $c$ of size $s_c$ at vertex $v$ to form $G+c$.
\begin{enumerate}
\item \label{lemma:cycle1} If $d_w\geq2$ for all vertices $w$ adjacent to $v$,
\[R(G+c) - R(G) \geq \frac{\sqrt{d_v+2}-\sqrt{d_v}}{\sqrt 2} + \frac{s_c-2}{2}\] 
with equality if all $d_w=2$.
\item \label{lemma:cycle2} If $d_v \geq 2$ and $d_w \leq d_v+2$ for all vertices $w$ adjacent to $v$, 
\[R(G+c) - R(G) \leq \frac{s_c-1}{2}\] 
with equality iff $d_v=2$ and all $d_w=4$.
\item \label{lemma:cycle3} If $d_v = 1$ and $2 \leq d_w \leq 3$ for the vertex $w$ adjacent to $v$, 
\[R(G+c) - R(G) < \frac{s_c-1}{2} + 0.075.\] 
\end{enumerate}
\end{lemma}

\begin{proof}
Define $S$ as in Lemma~\ref{lemma:leaf}. Observe
\[R(G+c) - R(G) = (s_c-2)\left(\frac{1}{2}\right)+2\left(\frac{1\phantom{.}}{\sqrt{2(d_v+2)}}\right)-S\left(\frac{1}{\sqrt{d_v}}-\frac{1}{\sqrt{d_v+2}}\right).\]
If $d_w\geq 2$, then $S \leq \frac{d_v}{\sqrt 2}$, and
\begin{multline*}R(G+c) - R(G) \geq \frac{s_c-2}{2}+\sqrt{\frac{2}{d_v+2}}-\frac{d_v}{\sqrt2} \left(\frac{1}{\sqrt{d_v}}-\frac{1}{\sqrt{d_v+2}}\right)\\
= \frac{\sqrt{d_v+2}-\sqrt{d_v}}{\sqrt2}+\frac{s_c-2}{2}\end{multline*}
with equality when all $d_w=2$. 
\end{proof}

\section{The radius and diameter of a cactus}\label{sec:rd_lemmata}

We will be studying radius and diameter under the addition of cycles or bridges, so we will need some a few facts for our later study. 

\begin{lemma}\label{lemma:centers}
All centers in a graph are contained in the same block.
\end{lemma}

\begin{proof}
If two centers $a_1$ and $a_2$ are not in the same block, then find a shortest path between $a_1$ and $a_2$; it must pass through some articulation point $a$. Then removing this articulation point would separate the graph into components, one containing $a_1$, one containing $a_2$, and possibly some additional components. Partition the vertices $V-a$ along these lines: a set $V_1$ for the component of $a_1$, a set $V_2$ for the component of $a_2$, and a set $V_0$ (possibly empty) for the remaining vertices. The vertices in $V_1$ set are closer to $a$ than they are to $a_2$, those in $V_2$ likewise are closer to $a$ than they are to $a_1$, and the remaining in $V_0$ are closer to $a$ than to either $a_i$; therefore, $a$ has eccentricity lower than either $a_i$, which is a contradiction.
\end{proof}

Consider a graph and select a special \emph{central block}: if there are multiple centers, select the block containing them; if there is a a unique center which is not an articulation point, select its block; if there is a the unique center which is an articulation point, select one adjoining block which contains an edge of a shortest path from that center to some vertex of distance $r$. 

\begin{lemma}\label{lemma:rd_vs_n}
Let $G$ be a cactus with $n$ vertices, $k$ cycles, and $b$ bridges. Let $d$ be its diameter and $r$ be its radius. Then
\[d \leq \frac{n+k+b-1}{2}\]
with equality if the BC-tree is a path and $G$ is longitudinally symmetric, and
\[r = \frac{n-k+1}{2}\]
if $G$ is an even cycle. Else 
\[r \leq \frac{n-k}{2}\]
with equality if $G$ is an even cycle with a single pendant; if $G$ consists of an even cycle and a triangle which share a vertex; or if $G$ has BC-tree a path and no cycles larger than triangles. In fact, if the central block has $m \geq 2$ articulation points, then 
\[r \leq \frac{n-k-m+2}{2}\]
with equality when the BC-tree is starlike with root corresponding to the central node and at least $m-2$ of the adjacent blocks being pendants.
\end{lemma}

Note that this gives some insight into the traditional bound of 
\[\frac{d}{2} \leq r \leq d.\]
with the first an equality for an odd path and the second for an even cycle. On the odd path, the bounds of the lemma become $d \leq n-1$ and $r < \frac{n+1}{2}$, that is, $r \leq \frac{n-1}{2}$, and both are sharp since $d = 2r = n-1$. For the even cycle, where $k=1$, $b=0$, the lemma gives $d \leq \frac{n}{2}$ and $r \leq \frac{n}{2}$, again, both of which are sharp since $d = r = \frac{n}{2}$. 

\begin{proof}
Consider two vertices $u$ and $v$ of maximum distance in the graph. Then the path between them passes through at most half of the vertices in each cycle as well as possibly all the bridges, so
\[d \leq b + \sum_c \frac{s_c}{2}.\]
Note that the total number of edges in the graph may be calculated as either $\sum s_c + b$ or as $n+k-1$, which gives the desired bound on $d$.

Similarly, as we have already commented, even cycles satisfy $n= 2r + k - 1$, odd cycles and even paths $n = 2r + k$, and odd paths $n = 2r + k + 1$. Note that adding a block to $B$ increases $n-k$ by 1 if it is a pendant or triangle and more than 1 if a larger cycle. Therefore, any other graph where $r(G) = r(B)$ for some block $B$ also satisfies the lemma inductively (note $n = 2r + k$ is only obtained if $B$ is an even cycle and $G$ is $B$ with a pendant or a triangle added), and any graph where $r(B) = r(G) - 1$ does as well (note $n = 2r + k$ is unobtainable since $G$ must have at least 1 more block than $B$ which is a cycle on at least 4 vertices, or else at least 2 more blocks than $B$). This has accounted for all cases with $m < 2$. 

For a generic cactus, consider a center $a$. Then there is at least one vertex $u$ of distance $r$ from $a$ and at least one other vertex $v$ of distance $r$ or $r-1$. Let $P$ be a shortest path from $u$ to $v$. If $P$ contains $a$, then 
\[\dist(u,v) = \dist(u,a)+\dist(a,v) \geq 2r-1,\]
 with equality if $P$ is an even path. Therefore, $P$ contains at least $2r$ vertices, and it misses at least one vertex from every cycle (else there are shorter paths between $u$ and $v$). In fact, it misses at least an additional $m-2$ vertices: by construction, $P$ passes through the central block. If that block has $m$ articulation points, then the vertex set may be divided into subsets, those from the central block and then $m$ more sets depending on which articulation point separates them from the central block. The path $P$ goes through the central block and at most two of these other subsets, so it misses at least one additional vertex from each of the $m-2$ other subsets. Thus, $n \geq 2r + k + m - 2$, with equality when $P$ is an even path, the cycles are 3-cycles and have one edge in $P$, and when any the articulation points on the central cycle (other than the one(s) on $P$) separates the cycle from a single pendant.

If $a$ is not on $P$, then $P$ must still share at least one edge with some block containing $a$. If not, then $a$ would not actually have the lowest eccentricity; let $b$ be a vertex one step closer to $u$, and let $B$ be the block containing $a$ and $b$. All the other vertices which are within $r-2$ of $a$ are still within $r-1$ of $b$. Since any vertex of distance at least $r-1$ from $a$ has the property that its path to $u$ does not share an edge with $B$, then this means that, if we remove from the graph all the edges of $B$, then these vertices will still be in the same connected component as $u$. There is some unique articulation point $u'$ where this component intersects $B$; since $b$ is closer to $u$ than $a$ is, it must also be closer to $u'$, and any vertex in this component is therefore within distance $r-1$ of $b$. Therefore, $b$ has lower eccentricity than $a$.

Assume $a$ is in a block $B$ that shares an edge with $P$ but that $a$ is not on $P$. Then $B$ must be a cycle rather than a bridge, and $P$ must enter the block at some articulation point $u'$ and exit it at some articulation point $v'$ (note $u$ and $v$ cannot be on $B$ themselves since $r(B) \leq r(G) - 2$; in fact, each of them is distance at least 2 from $B$, so $\dist(u,v) \geq 5$). There are two paths between $u'$ and $v'$ inside $B$, call them $P_a$ (containing $a$) and $P_{-a}$. Then $P_{-a}$ must be shorter (else $P$ could contain $a$). Consider a vertex $b$ which is on the path $P_{-a}$ (and therefore also $P$). Without loss of generality, assume $b$ is not a center, so there must be some other vertex $w$ whose distance to $b$ is at least $r+1$, with corresponding $w'$ (the vertex where a path from $w$ to $b$ first enters the cycle $B$; once more, it is an articulation point since $w$ must be at least distance 3 from $B$). Now assume $w'$ is neither $u'$ nor $v'$ (we may choose $b$ so that $\dist(b,u') \leq \dist(a,u')$ and $\dist(b,v') \leq \dist(a,v')$, so $\dist(b,u) \leq r$ and $\dist(b,v) \leq r$; therefore, if $w'$ must be $u'$ or $v'$, then that means all vertices of distance at least $r+1$ from $b$ are also distance at most $r$ from $b$, which is a contradiction). Now let us add up the number of edges in a set of paths: a shortest path from $u$ to $v$; that path, but altered to go the other way around $B$ to contain $a$; a shortest path from $b$ to $w$; and that path, but again altered to go the other way around $B$. These paths respectively have lengths at least $\dist(u,v)$; $2r-1$; $r+1$; and $r+1$. That is, the total is at least $\dist(u,v) + 4r + 1 \geq 4r+6$. When counting these edges, we have double-counted each edge, so these paths therefore actually contain at least $2r+3$ distinct edges. Since they constitute a subgraph with one cycle, they thus contain at least $2r+3$ distinct vertices. This subgraph must miss at least one vertex from each of the $k-1$ cycles other than $B$, and it must miss at least one additional vertex for each of the other $m-3$ articulation points on $B$ (note $B$ meets the definition of a central block), so $n \geq 2r + k + m - 1$.
\end{proof}

Finally, we define two special subgraphs which realize the radius and diameter of the overall graph. First, consider two vertices $u$ and $v$ of maximal distance and a shortest path $P$ between them. Let $H_d$ be the subgraph consisting of all the blocks which contain an edge of $P$. Next, if a center $a$ is contained on the path $P$, then $H_r=H_d$ will also realize the radius. If not, then find the central block $B$. For each articulation point $v'$ in this block, consider the set of all vertices it separates from $B$ and identify one of them of maximal distance from $B$, call it $v$. Let $T$ be a minimal tree containing all these vertices $v$, and let $H_r$ be the subgraph consisting of all blocks which contain an edge of $T$. See Figures~\ref{fig:ntc_d} and \ref{fig:cactus_d} for examples of $H_d$ and Figures~\ref{fig:ntc_r} and \ref{fig:cactus_r} for examples of $H_r$.

\begin{lemma}\label{lemma:realize_rd}
Let $G$ be a cactus with diameter $d$ and radius $r$. Let $H_d$ and $H_r$ be the subgraphs defined above. Then $H_d$ has the same diameter as $G$, and $H_r$ has the same centers, central block, and radius as $G$.
\end{lemma}

\begin{proof}
The diameter of $H_d$ is same as that of $d(G)$ by construction; there are two vertices of distance $d=d(G)$, and no vertices are more distant in $H_d$ than in $G$ because any shortest path between them in $G$ goes through the same blocks as $P$, and, therefore, is contained wholly in $H_d$. Similarly, if some center $a$ is contained in the path $P$, then no other vertex will lower eccentricity within $H_d$, so the same subgraph realizes the radius as well. 

On the other hand, if $P$ does not contain any centers, and $H_r$ is defined from the minimal tree $T$, then consider constructing $H_r$ by successively removing cycles (with single articulation points) or blocks from $G$. Removing one such will not alter the eccentricity of any center or any other vertex contained in $B$; any path from such a vertex to any other vertex outside $B$ must pass through one of $B$'s articulation points, and we have preserved a vertex at maximal distance past each of these articulation points. Additionally, it cannot alter the eccentricity of any other vertex $v$: its distance from the vertices in $B$ or on the other side of $B$ has not been reduced, and one of these vertices must have previously been at distance greater than $r$. (If not, then consider the articulation point $v'$ on $B$ which separates $v$ from $B$, and consider the vertices $V_B$, those in the component of $G-v$ which contains $B$, and $V_{-B}$, the rest. If $v$ is has eccentricity less than $r$, then it is distance at most $r-1$ from any vertex in $V_B$, which means $v'$ is distance at most $r-2$ from these vertices. If $v'$ is a unique center, then we had constructed $B$ so that one of these vertices was in fact distance $r$ from $v'$; on the other hand, if $v'$ was not a unique center, then there was at least one other center in $B$ of distance at most $r$ from any vertex in $V_{B}$, so $v'$ is also distance at most $r-1$ from these vertices. Therefore, the eccentricity of $v'$ in $G$ was, in fact, less than $r$).
\end{proof}

\section{Randi\'c index vs. graph order/size/valency}\label{sec:osv}

\subsection{Trees} 

First, we may use Lemma~\ref{lemma:leaf} on its own to approximate $R$ for a tree:

\begin{thm}\label{thm:tree_osv}
Let $G$ be a tree. Then
\[R \geq 1+\sum_{v} \left(\sqrt{d_v}-1\right),\]
with equality if $G$ is the star. 
\end{thm}

Note that the first half of Corollary~\ref{cor:valency} follows immediately,
\[R \geq 1 - n + \sum_{v} \sqrt{d_v}.\]

\begin{proof}
Observe the inequality holds for an isolated vertex, where $R=0$. Adding a new pendant $u$ at some vertex $v$ increases its degree from $d_v$ to $d_v+1$ and changes $R$ by:
\[R(G + u) - R(G) \geq \sqrt{d_v+1}-\sqrt{d_v},\] 
as in Lemma~\ref{lemma:leaf}. This is an equality when $d_w = 1$ for all vertices $w$ adjacent to $v$, so the original formula is an equality only for the star, which we independently calculated in Example~\ref{ex:star}.
\end{proof}

This gives us the tools to bound $R$; for example, it immediately produces a theorem of Bollob\'as and Erd\"os, who proved the same result inductively adding pendants to the path on 3 vertices. 

\begin{cor}\cite[Theorem~3]{bollobas1998extremal}\label{cor:tree_extrema} Let $G$ be a tree on $n$ vertices. If $n=2$ and the graph is a single edge, then $R=1$; else,
\[R \geq \sqrt {n-1}\]
with the lower bound realized only by the star.
\end{cor}

\begin{proof}
Note $\sum d_v$ is fixed (it is twice the number of edges or $2n-2$), and $\sum \sqrt{d_v}$ is minimized when a single vertex is degree $n-1$ and the others are degree $1$, in which case $R = \sqrt{n-1}$. This lower bound is realized by the star (and, by the lemma, only the star).

\end{proof}

\subsection{Nontrivial cacti}

Next, we consider the case of a nontrivial cactus. As with the trees, we will begin with a single cycle $C_n$ and add cycles one by one by identifying one vertex in the new cycle with one vertex in the old graph. This theorem and its corollaries will be essential in our study of the radius and diameter of a nontrivial cactus, but the first theorem is true for cacti with bridges as well, so we state it as such.

\begin{thm}\label{thm:cactus_osv}
Let $G$ be a cactus with $n$ vertices, $k>0$ cycles, and $b$ bridges. Then
\[R \geq 1+ \sum_{c} \left(\frac{s_c}{2} - 1\right) + \sum_br \left(\sqrt 2 - 1\right) + \sum_{v} \left(\sqrt{\frac{d_v}{2}}-1\right),\]
where the sums run over cycles, bridges, and vertices, respectively. Equality occurs if $G$ is a nontrivial cactus with none of the articulation points adjacent or $G$ is the path on 3 vertices.
\end{thm}

We may equivalently say
\[\begin{split}
R &\geq \frac{1+n-k}{2} + b\left(\sqrt 2 - \frac{3}{2}\right) + \sum_{v} \left(\sqrt{\frac{d_v}{2}}-1\right) \\
&=  \frac{1-n-k }{2} + b\left(\sqrt 2 - \frac{3}{2}\right) + \sum_{v} \sqrt{\frac{d_v}{2}},
\end{split}\]
since the number of edges may be calculated as either $\sum s_c + e$ or $n+k-1$; see Corollary~\ref{cor:valency}. Note that this equation does not subsume Theorem~\ref{thm:tree_osv}; if we set $k=0$, we may get a stronger bound than is warranted for, ex, the star.

\begin{proof}
The bound is sharp for a single cycle, where $R=\frac{n}{2}$, but it is slightly weaker for a path. We may build any other $G$ by beginning with a cycle, then adding a cycle $c$ or a pendant at vertex $v$, then repeating as needed. If we are careful to add all cycles to a given articulation point before we add any pendants, then we may add the cycles according to Lemma~\ref{lemma:cycle},
\[R(G + c) - R(G) \geq \frac{\sqrt{d_v+2}-\sqrt{d_v}}{\sqrt 2} + \frac{s_c-2}{2}\]
with equality if all the vertices adjacent to $v$ are degree 2. We may then add the vertices according to Lemma~\ref{lemma:leaf}
\[R(G + uv) - R(G) \geq \sqrt{d_v+1}-\sqrt{d_v} \geq \sqrt\frac{d_v+1}{2}-\sqrt\frac{d_v}{2}+ \sqrt\frac{1}{2} - 1 + \sqrt 2 - 1\]
with equality if $v$ is a leaf, and its one adjacent vertex is also a leaf. 

The desired inequality follows and is sharp either if $G$ is a path on three vertices or if there were no pendants added (thus no bridges in $G$) and when cycles were not added adjacent to preexisting articulation points.
\end{proof}

Note that if $G$ is a nontrivial cactus, we may deduce the actual minimum value of $R$, and the result resembles the result for trees:

\begin{cor}\label{cor:ntc_osv}
Let $G$ be a nontrivial cactus on $n$ vertices with $k$ cycles. Then
\[R \geq \frac{n - k - 1}{2} + \sqrt k\]
with equality for a bouquet of cycles.
\end{cor}

\begin{proof}
Observe that $\sum \sqrt{d_v}$ attains its minimum value when a single vertex has maximal degree ($d_v = 2k$) and the others are all minimum degree ($d_v=2$). 
\end{proof}

In the case of a chemical nontrivial cactus, the bound from the theorem becomes especially nice. This result will be essential later in Theorems~\ref{thm:ntc_d} and \ref{thm:ntc_r}.

\begin{cor}\label{cor:chemical_ntc_osv}
If $G$ is a chemical nontrivial cactus on $n$ vertices with $k$ cycles,
\[R \geq \frac{n}{2} - (k-1)\left(\frac{3}{2} - \sqrt 2\right).\]
with equality when none of the articulation points are adjacent.
\end{cor}

\begin{proof}
Since the BC-tree is a path, the cycles and bridges occur in a chain, and there are $k-1$ articulation points which must therefore be degree 4, and all remaining vertices must be degree 2.
\end{proof}

\subsection{Cacti}
A similar result to Corollary~\ref{cor:chemical_ntc_osv} could be derived for generic cacti, but it will not be sharp since Theorem~\ref{thm:cactus_osv} is not sharp if the graph has any bridges. Instead, we separately develop a special result for use in the proof of Theorem~\ref{thm:cactus_r}; it is about 0.2 weaker than the corresponding result for nontrivial cacti, and it is sharp. Observe that the graph of Example~\ref{ex:lollipop} realizes this bound.

\begin{thm}\label{thm:star_cactus_osv}
Let $G$ be a cactus (with $k>0$ cycles and $b>0$ bridges and/or pendants) whose BC-tree is starlike with root of degree $m$. Then 
\[R \geq \frac{n}{2} - (k-1)\left(\frac{3}{2} - \sqrt 2 \right) - m\left(\frac{3}{2} - \frac{1}{\sqrt 3} - \sqrt\frac{2}{3}\right).\]
with equality when there are $m$ pendants and when no articulation points are adjacent.
\end{thm}

In particular, if the BC-tree has a central block with only two articulation points (as it will be when, ex, the central block is a bridge), then
\[R \geq \frac{n}{2} - (k-1)\left(\frac{3}{2} - \sqrt 2 \right) - 3 + \frac{2}{\sqrt 3} + 2\sqrt\frac{2}{3}.\]

\begin{proof}
If the BC-tree is starlike, its root of degree $r$ corresponds to a block with $r$ articulation points, and its $r$ leaves correspond to pendants or cycles with single articulation points. All other blocks in the graph are bridges or cycles with a pair of articulation points. Note that $r=0$ implies the graph has no articulation points, and $r > 2$ implies the central block must be a cycle.

We will consider the individual contribution of each block to $R$. Let 
\[w_{x,y} = \left(\frac{1}{\sqrt x}-\frac{1}{\sqrt y}\right)^2\]
and recall the definition of $R$ given in Equation~\eqref{eqn:weight} for a connected graph:
\[n-2R = \sum_{uv}w_{d_u,d_v}\]
where the sum runs over all edges $uv$. A cycle $c$ with $m_c$ non-adjacent articulation points, of which $t_c$ are degree 3 and the others degree 4, will contribute to $2R-n$ by exactly:
\[2t_cw_{2,3}+2(m_c-t_c)w_{2,4}.\]
Note that, if articulation points are adjacent, the cycle actually contributes by a greater amount than the above equation indicates: for any pair of the same degree $d_v$ which are adjacent, this equation produces a value which is too small by $2w_{2,d_v}$ and, for every pair of a degree 3 and degree 4 which are adjacent, by $w_{2,3}+w_{2,4}-w_{3,4}>0$. 

Next, consider the contribution of a block $e$ which is a single edge. If it is a pendant, it contributes to $2R-n$ by exactly
\[t_ew_{1,3}+(1-t_e)w_{1,2},\]
and, if it is a bridge, then it contributes by
\[\left.\begin{cases} 0 & t_e = 0\\ w_{2,3} & t_e = 1\\ 0 & t_e = 2\end{cases}\right\} \leq t_ew_{2,3}\]
with equality when $t_e < 2$.

Now consider the contribution of all blocks in the cactus, cycles, bridges, and pendants (let $c$ run over cycles, $e$ over bridges/pendants, and $e_p$ separately over pendants):
\[n-2R \leq \sum_c 2m_cw_{2,4} + 2t_c(w_{2,3}-w_{2,4}) + \sum_e t_ew_{2,3} + \sum_{e_p} w_{1,2} + t_{e_p}(w_{1,3}-w_{1,2}-w_{2,3})\]
Let $t$ be the number of degree 3 vertices; since each is part of exactly one cycle and exactly one bridge/pendant, this quantity is
\[= t(3w_{2,3}-2w_{2,4}) + \sum_c 2m_cw_{2,4} + \sum_{e_p} w_{1,2} + t_{e_p}(w_{1,3}-w_{1,2}-w_{2,3})\]
Let $p$ be the number of pendants and $p'$ the number of pendants attached to cycles (i.e., with $t_{e_p} = 1$). Let $k$ be the number of cycles, of which one has $m$ articulation points ($m$ may be 2), $m-p$ have a single articulation point, and the rest have 2 articulation points. Note also $2w_{2,4} =w_{1,2}$. Now:
\[= t(3w_{2,3}-w_{1,2}) + 2(p+k-1)w_{1,2} + p'(w_{1,3}-w_{1,2}-w_{2,3})\]

Observe this quantity is decreasing in $t$ and increasing in $p$ and $p'$, so it is maximized if we rearrange the blocks to maximize the number of pendants and the number of pendants attached to cycles as well as minimize the number of adjacent cycle/bridge or cycle/pendant pairs. If the central cycle is a cycle, and if the total number of bridges $b$ is $b > m$, then we may indeed fix the central cycle and rearrange the other blocks so that there are $m$ pendants, all but one attached to a cycle, and the remaining pendant attached to a path of the remaining bridges, that is, $p=t=m$ and $p'=m-1$; then the bound becomes:
\[(2k-1)w_{1,2} + (m-1)w_{1,3} + (2m+1)w_{2,3} \]
and, if $b \leq m$, then $p=p'=t=b$, then this value is:
\[2(k-1)w_{1,2} + bw_{1,3} + 2bw_{2,3} \]
so the maximum value occurs when $b=m$, and thus
\[n-2R \leq 2(k-1)w_{1,2} + mw_{1,3} + 2mw_{2,3}.\]
Similarly, if the central block is a bridge (but the graph is not a tree), then the BC-tree is a path with $m=2$. In this case, we do not fix the central block, but we may still rearrange all the blocks so that, if $b > 2$, then $p=t=2$ and $p'=1$, else $b \leq 2$ and $p=p'=t=b$.

To attain equality, we must have the articulation points within any cycle not adjacent. We must also have $m$ pendants, each attached to a cycle, and no other bridges (which will follow if no articulation points are adjacent). 
\end{proof}

\section{Randi\'c index vs. graph radius/diameter}\label{sec:rd}

We now address the relationship between $R$ and the radius and diameter. This will involve the valency theorems above, although it will also require a fairly nuanced analysis of radius.

\subsection{Trees}
We approach first the question of the tree, where we will reprove Theorem~\ref{thm:d} (although substituting $e=n-1$ for consistency with later results) and verify Conjecture~\ref{conj:r}. While the diameter result is already known and the radius result not that difficult, this will gives us the pattern for later proofs. 

We will build a subgraph with maximal radius and diameter (as in the thick black subgraphs of Figure~\ref{fig:tree}), check $R$ against $r$ and $d$, and then add pendants to construct the entire graph.

\begin{figure}
\begin{subfigure}[scale=.5]{0.4\textwidth}
\begin{tikzpicture}
\draw[black, thick] (-1.3,-0.5) -- (-0.5, -0.5);
\draw[gray, thick] (-0.5,0.5) -- (0, 0);
\draw[black, thick] (-0.5,-0.5) -- (0, 0);
\draw[black, thick] (0,0) -- (0.8, 0);
\draw[black, thick] (0.8,0) -- (1.3, 0.5);
\draw[gray, thick] (1.3,0.5) -- (0.8, 1);
\draw[black, thick] (1.3,0.5) -- (2.7, 0.5);
\draw[gray, thick] (2.1,0.5) -- (2.6, 1);
\draw[gray, thick] (2.1,0.5) -- (2.6, 0);
\draw[black, thick] (2.1,0.5) -- (2.9, 0.5);
\filldraw[black] (-1.3,-0.5) circle (2pt)node[anchor=east] {$u$};
\filldraw[black] (-0.5,-0.5) circle (2pt);
\filldraw[gray] (-0.5,0.5) circle (2pt);
\filldraw[black] (0,0) circle (2pt);
\filldraw[black] (0.8,0) circle (3pt);
\filldraw[white] (0.8,0) circle (2pt);
\filldraw[black] (1.3,0.5) circle (2pt);
\filldraw[gray] (0.8,1) circle (2pt);
\filldraw[black] (2.1,0.5) circle (2pt);
\filldraw[gray] (2.6, 1) circle (2pt);
\filldraw[gray] (2.6, 0) circle (2pt);
\filldraw[black] (2.9, 0.5) circle (2pt)node[anchor=west] {$v$};
\end{tikzpicture}
\subcaption{}\label{fig:tree1}
\end{subfigure}
\begin{subfigure}[scale=0.5]{0.4\textwidth}
\begin{tikzpicture}
\draw[gray, thick] (-0.5,0.5) -- (0, 0);
\draw[black, thick] (-0.5,-0.5) -- (0, 0);
\draw[black, thick] (0,0) -- (0.8, 0);
\draw[black, thick] (0.8,0) -- (1.3, 0.5);
\draw[gray, thick] (1.3,0.5) -- (0.8, 1);
\draw[black, thick] (1.3,0.5) -- (2.1, 0.5);
\draw[gray, thick] (2.1,0.5) -- (2.6, 1);
\draw[gray, thick] (2.1,0.5) -- (2.6, 0);
\draw[black, thick] (2.1,0.5) -- (2.9, 0.5);
\filldraw[black] (-0.5,-0.5) circle (2pt)node[anchor=east] {$u$};
\filldraw[gray] (-0.5,0.5) circle (2pt);
\filldraw[black] (0,0) circle (2pt);
\filldraw[black] (0.8,0) circle (3pt);
\filldraw[white] (0.8,0) circle (2pt);
\filldraw[black] (1.3,0.5) circle (3pt);
\filldraw[white] (1.3,0.5) circle (2pt);
\filldraw[gray] (0.8,1) circle (2pt);
\filldraw[black] (2.1,0.5) circle (2pt);
\filldraw[gray] (2.6, 1) circle (2pt);
\filldraw[gray] (2.6, 0) circle (2pt);
\filldraw[black] (2.9, 0.5) circle (2pt)node[anchor=west] {$v$};
\end{tikzpicture}
\subcaption{}\label{fig:tree2}
\end{subfigure}

\caption{The diameter of a tree is realized by a subgraph (thick black edges), a path between two vertices $u$, $v$ of maximal distance. (A) If the path is odd, its middle vertex (the hollow vertex) will be the unique center, the vertex with maximum eccentricity, and $d=2r$. (B) If the path is even, the two vertices on either side of its middle edge will be the two centers, and $d=2r-1$.}\label{fig:tree}
\end{figure}
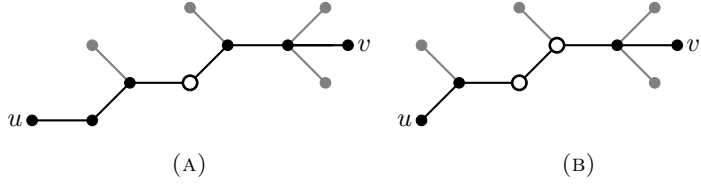

\begin{thm}\label{thm:tree_rd} Let $G$ be a tree on $n\geq 2$ vertices with $e$ edges. Then
\begin{enumerate}
\item\cite[Theorem]{yang2011randic}
\[\begin{split}
R - d &\geq -\frac{e}{2}+\sqrt 2 -1 \\
R - \frac{d}{2} &\geq \sqrt 2 - 1 \\
\frac{R}{d} &\geq \frac{n-3+2\sqrt 2}{n+e-1}
\end{split}\]
with equality if $G$ is a path with $n >2$. 
\item Additionally, if $G$ is an even path on more than 2 vertices,
\[R - r = \sqrt 2 - \frac{3}{2},\]
and, otherwise,
\[R - r \geq 0\]
with equality for the path on 2 vertices.
\end{enumerate}
\end{thm}

We reprove this theorem to demonstrate techniques we will use for cacti.

Note that $R$ is approximately $r - 0.1$ for even paths, $r + 0.4$ for odd paths, and increasing as the tree branches further.

\begin{proof}
By Example~\ref{ex:path}, all paths satisfy the theorem: $d=n-1$ and $r = \left\lfloor \frac{n}{2} \right\rfloor$, while $R = 1$ for $P_2$ and $R=\frac{n-3}{2} + \sqrt 2$ for longer paths. In fact, any tree with diameter 1 is necessarily a path and thus satisfies the theorem as well. 

If $G$ is not a path, then choose two vertices $u, v$ of maximum distance and find the shortest path $P$ between them (ex, the thick black subgraphs in Figure~\ref{fig:tree}). This is a path of length $d > 1$; in particular, it is itself a graph with the same diameter as $G$, and it obeys the bounds on $d$. Next, add pendants to fill out the tree. By Lemma~\ref{lemma:leaf}, adding a pendant increases $R$ while leaving $d$ fixed by assumption, and the right-hand sides are non-increasing as functions of $n$, so all three inequalities on diameter are inductively satisfied.

Similarly, if the tree has even diameter like $T_{1}$, it has a unique center $a$ and $d=2r$, and its maximal diameter subgraph $P$ is an even path. Then $P$ satisfies the theorem, and, once more, adding pendants to fill out $G$ increases $R$ without altering $r$. 

On the other hand, if the tree has odd diameter like $T_{2}$, then it has two adjacent centers $a_1$ and $a_2$, its diameter is $d=2r-1$, and $P$ is an even path that contains both centers. Alas, the radius obeys $R \approx r-0.1$; however, adding a single pendant $l$ at some vertex $v$ (which has degree 2 by assumption, else adding it would alter $d$) will increase $R$:
\[R(P + l) \geq R(P) + \sqrt 3 - \sqrt 2 = r + \sqrt 3 - \frac{3}{2} > r.\]
After this, adding additional pendants to fill out $G$ will only increase $R$ further without altering $r$.
\end{proof}

\subsection{Nontrivial cacti}

A similar technique works for a nontrivial cactus. 

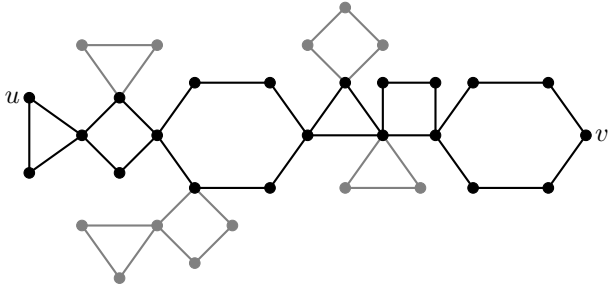
\begin{figure}
\begin{tikzpicture}
\draw[gray, thick] (0.4,-1.2) -- (-0.6,-1.2);
\draw[gray, thick] (0.4,-1.2) -- (-0.1,-1.9);
\draw[gray, thick] (-0.6,-1.2) -- (-0.1,-1.9);

\draw[gray, thick] (-0.1,0.5) -- (0.4,1.2);
\draw[gray, thick] (-0.1,0.5) -- (-0.6,1.2);
\draw[gray, thick] (0.4,1.2) -- (-0.6,1.2);

\draw[black, thick] (-0.6,0) -- (-1.3,0.5);
\draw[black, thick] (-0.6,0) -- (-1.3,-0.5);
\draw[black, thick] (-1.3,0.5) -- (-1.3,-0.5);

\draw[black, thick] (-0.1,0.5) -- (0.4,0);
\draw[black, thick] (-0.1,-0.5) -- (0.4,0);
\draw[black, thick] (-0.1,0.5) -- (-0.6,0);
\draw[black, thick] (-0.1,-0.5) -- (-0.6,0);

\draw[gray, thick] (0.9,-0.7) -- (0.4,-1.2);
\draw[gray, thick] (0.9,-0.7) -- (1.4,-1.2);
\draw[gray, thick] (0.9,-1.7) -- (0.4,-1.2);
\draw[gray, thick] (0.9,-1.7) -- (1.4,-1.2);

\draw[black, thick] (0.4,0) -- (0.9,0.7);
\draw[black, thick] (0.4,0) -- (0.9,-0.7);
\draw[black, thick] (0.9,0.7) -- (1.9,0.7);
\draw[black, thick] (0.9,-0.7) -- (1.9,-0.7);
\draw[black, thick] (1.9,0.7) -- (2.4,0);
\draw[black, thick] (1.9,-0.7) -- (2.4,0);

\draw[black, thick] (2.4,0) -- (2.9,0.7);
\draw[black, thick] (2.4,0) -- (3.4,0);
\draw[black, thick] (2.9,0.7) -- (3.4,0);

\draw[gray, thick] (2.4,1.2) -- (2.9,0.7);
\draw[gray, thick] (3.4,1.2) -- (2.9,0.7);
\draw[gray, thick] (2.4,1.2) -- (2.9,1.7);
\draw[gray, thick] (3.4,1.2) -- (2.9,1.7);

\draw[black, thick] (3.4,0) -- (3.4, 0.7);
\draw[black, thick] (3.4,0) -- (4.1, 0);
\draw[black, thick] (3.4,0.7) -- (4.1, 0.7);
\draw[black, thick] (4.1,0.7) -- (4.1, 0);

\draw[gray, thick] (3.4,0) -- (3.9, -0.7);
\draw[gray, thick] (3.4,0) -- (2.9, -0.7);
\draw[gray, thick] (3.9,-0.7) -- (2.9, -0.7);

\draw[black, thick] (4.1,0) -- (4.6, 0.7);
\draw[black, thick] (4.1,0) -- (4.6, -0.7);
\draw[black, thick] (5.6,0.7) -- (4.6, 0.7);
\draw[black, thick] (5.6,-0.7) -- (4.6, -0.7);
\draw[black, thick] (5.6,0.7) -- (6.1, 0);
\draw[black, thick] (5.6,-0.7) -- (6.1, 0);

\filldraw[gray] (-0.1,-1.9) circle (2pt);
\filldraw[gray] (-0.6,-1.2) circle (2pt);
\filldraw[gray] (-0.6,1.2) circle (2pt);
\filldraw[gray] (0.4,1.2) circle (2pt);
\filldraw[black] (-1.3,0.5) circle (2pt)node[anchor=east] {$u$};
\filldraw[black] (-1.3,-0.5) circle (2pt);
\filldraw[black] (-0.1,0.5) circle (2pt);
\filldraw[black] (-0.1,-0.5) circle (2pt);
\filldraw[black] (-0.6,0) circle (2pt);
\filldraw[gray] (0.4,-1.2) circle (2pt);
\filldraw[gray] (1.4,-1.2) circle (2pt);
\filldraw[gray] (0.9,-1.7) circle (2pt);
\filldraw[black] (0.4,0) circle (2pt);
\filldraw[black] (0.9,0.7) circle (2pt);
\filldraw[black] (0.9,-0.7) circle (2pt);
\filldraw[black] (1.9,0.7) circle (2pt);
\filldraw[black] (1.9,-0.7) circle (2pt);
\filldraw[black] (2.4,0) circle (2pt);
\filldraw[black] (2.9,0.7) circle (2pt);
\filldraw[black] (3.4,0) circle (2pt);
\filldraw[black] (3.4,0.7) circle (2pt);
\filldraw[black] (4.1,0.7) circle (2pt);
\filldraw[gray] (3.9,-0.7) circle (2pt);
\filldraw[gray] (2.9,-0.7) circle (2pt);
\filldraw[gray] (2.9,1.7) circle (2pt);
\filldraw[black] (4.1,0) circle (2pt);
\filldraw[gray] (2.4,1.2) circle (2pt);
\filldraw[gray] (3.4,1.2) circle (2pt);
\filldraw[black] (4.6,0.7) circle (2pt);
\filldraw[black] (4.6,-0.7) circle (2pt);
\filldraw[black] (5.6,0.7) circle (2pt);
\filldraw[black] (5.6,-0.7) circle (2pt);
\filldraw[black] (6.1,0) circle (2pt)node[anchor=west] {$v$};
\end{tikzpicture}
\caption{A nontrivial cactus, a graph whose blocks are all cycles. The vertices $u$ and $v$ are of maximal distance, and the diameter is realized by the thick black subgraph $H_d$, the subgraph of all cycles containing these two points and the path between them, in other words, the smallest connected nontrivial cactus subgraph containing both of these points.}\label{fig:ntc_d}
\end{figure}

\begin{thm}\label{thm:ntc_d}
Let $G$ be a nontrivial cactus on $n$ vertices with $k$ cycles. Then
\[\begin{split}
R - d &\geq -(k - 1)\left(2 - \sqrt 2 \right)\\
R-\frac{d}{2} &\geq \frac{n}{4} -(k-1)\left(\frac{7}{4} - \sqrt 2\right)\\
\end{split}\]
with equality if the graph has BC-tree a path and is longitudinally symmetric.
\end{thm}

This improves on two of Yang and Lu's bounds of Theorem~\ref{thm:d} for nontrivial cacti. We further have a provisional improvement on the last bound (see Conjecture~\ref{conj:d_new}); it follows immediately from the work below for nontrivial cacti whose BC-trees are paths, but it is difficult to see how to navigate the inductive step:
\[\frac{R}{d} \geq \frac{n - (k-1)(3-2\sqrt 2)}{n+k-1}\]

\begin{proof}
Example~\ref{ex:cycle} covers the case where $G$ is a single cycle $C_n$: $R=\frac{n}{2}$ and $d=r=\left\lfloor\frac{n}{2}\right\rfloor$. 

For a general nontrivial cactus, we will identify a subgraph $H_d$ with the same diameter as $G$ (see Lemma~\ref{lemma:realize_rd}, seen as the dark subgraph Figure~\ref{fig:ntc_d}), verify the bounds there, then add additional cycles as needed to fill out $G$. Note $H_d$ is a nontrivial cactus and, in fact, has BC-tree a path and is a chemical graph. 

Apply Lemma~\ref{lemma:rd_vs_n} and Corollary~\ref{cor:chemical_ntc_osv}; the bounds follows, and equality occurs if the BC-tree is a path and the graph is longitudinally symmetric (note that this implies the articulation points are not adjacent).

Next, add cycles to fill out the original graph $G$. By Lemma~\ref{lemma:cycle}, this increases $R$ while leaving $d$ fixed by assumption, so the inequality on $R-d$ is immediately satisfied. For $R-\frac{d}{2}$, observe that $R$ increases by at least $\frac{s_c-2}{2}$, whereas the right-hand side increases by $\frac{s_c-1}{4} - \frac{7}{4} + \sqrt 2,$ which is smaller since $s_c \geq 3$.
\end{proof}


The case for radius is slightly more complicated because we will need a different subgraph to capture the radius. This will verify Conjecture~\ref{conj:r} for nontrivial cacti.

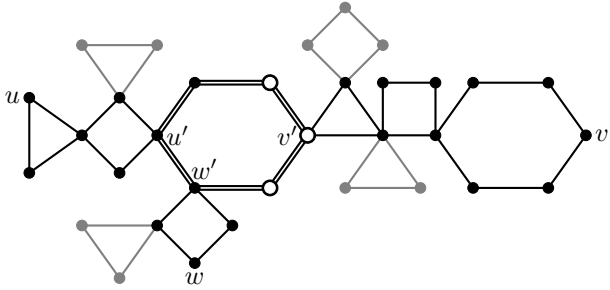
\begin{figure}
\begin{tikzpicture}
\draw[gray, thick] (0.4,-1.2) -- (-0.6,-1.2);
\draw[gray, thick] (0.4,-1.2) -- (-0.1,-1.9);
\draw[gray, thick] (-0.6,-1.2) -- (-0.1,-1.9);

\draw[gray, thick] (-0.1,0.5) -- (0.4,1.2);
\draw[gray, thick] (-0.1,0.5) -- (-0.6,1.2);
\draw[gray, thick] (0.4,1.2) -- (-0.6,1.2);

\draw[black, thick] (-0.6,0) -- (-1.3,0.5);
\draw[black, thick] (-0.6,0) -- (-1.3,-0.5);
\draw[black, thick] (-1.3,0.5) -- (-1.3,-0.5);

\draw[black, thick] (-0.1,0.5) -- (0.4,0);
\draw[black, thick] (-0.1,-0.5) -- (0.4,0);
\draw[black, thick] (-0.1,0.5) -- (-0.6,0);
\draw[black, thick] (-0.1,-0.5) -- (-0.6,0);

\draw[black, thick] (0.9,-0.7) -- (0.4,-1.2);
\draw[black, thick] (0.9,-0.7) -- (1.4,-1.2);
\draw[black, thick] (0.9,-1.7) -- (0.4,-1.2);
\draw[black, thick] (0.9,-1.7) -- (1.4,-1.2);

\draw[black, double, thick] (0.4,0) -- (0.9,0.7);
\draw[black, double, thick] (0.4,0) -- (0.9,-0.7);
\draw[black, double, thick] (0.9,0.7) -- (1.9,0.7);
\draw[black, double, thick] (0.9,-0.7) -- (1.9,-0.7);
\draw[black, double, thick] (1.9,0.7) -- (2.4,0);
\draw[black, double, thick] (1.9,-0.7) -- (2.4,0);

\draw[black, thick] (2.4,0) -- (2.9,0.7);
\draw[black, thick] (2.4,0) -- (3.4,0);
\draw[black, thick] (2.9,0.7) -- (3.4,0);

\draw[gray, thick] (2.4,1.2) -- (2.9,0.7);
\draw[gray, thick] (3.4,1.2) -- (2.9,0.7);
\draw[gray, thick] (2.4,1.2) -- (2.9,1.7);
\draw[gray, thick] (3.4,1.2) -- (2.9,1.7);

\draw[black, thick] (3.4,0) -- (3.4, 0.7);
\draw[black, thick] (3.4,0) -- (4.1, 0);
\draw[black, thick] (3.4,0.7) -- (4.1, 0.7);
\draw[black, thick] (4.1,0.7) -- (4.1, 0);

\draw[gray, thick] (3.4,0) -- (3.9, -0.7);
\draw[gray, thick] (3.4,0) -- (2.9, -0.7);
\draw[gray, thick] (3.9,-0.7) -- (2.9, -0.7);

\draw[black, thick] (4.1,0) -- (4.6, 0.7);
\draw[black, thick] (4.1,0) -- (4.6, -0.7);
\draw[black, thick] (5.6,0.7) -- (4.6, 0.7);
\draw[black, thick] (5.6,-0.7) -- (4.6, -0.7);
\draw[black, thick] (5.6,0.7) -- (6.1, 0);
\draw[black, thick] (5.6,-0.7) -- (6.1, 0);

\filldraw[gray] (-0.1,-1.9) circle (2pt);
\filldraw[gray] (-0.6,-1.2) circle (2pt);
\filldraw[gray] (-0.6,1.2) circle (2pt);
\filldraw[gray] (0.4,1.2) circle (2pt);
\filldraw[black] (-1.3,0.5) circle (2pt)node[anchor=east] {$u$};
\filldraw[black] (-1.3,-0.5) circle (2pt);
\filldraw[black] (-0.1,0.5) circle (2pt);
\filldraw[black] (-0.1,-0.5) circle (2pt);
\filldraw[black] (-0.6,0) circle (2pt);
\filldraw[black] (0.4,-1.2) circle (2pt);
\filldraw[black] (1.4,-1.2) circle (2pt);
\filldraw[black] (0.9,-1.7) circle (2pt)node[anchor=north] {$w$};
\filldraw[black] (0.4,0) circle (2pt)node[anchor=west] {$u'$};
\filldraw[black] (0.9,0.7) circle (2pt);
\filldraw[black] (0.9,-0.7) circle (2pt)node[anchor=south] {$\phantom{w}w'$};
\filldraw[black] (1.9,0.7) circle (3pt);
\filldraw[white] (1.9,0.7) circle (2pt);
\filldraw[black] (1.9,-0.7) circle (3pt);
\filldraw[white] (1.9,-0.7) circle (2pt);
\filldraw[black] (2.4,0) circle (3pt)node[anchor=east] {$v'$};
\filldraw[white] (2.4,0) circle (2pt);
\filldraw[black] (2.9,0.7) circle (2pt);
\filldraw[black] (3.4,0) circle (2pt);
\filldraw[black] (3.4,0.7) circle (2pt);
\filldraw[black] (4.1,0.7) circle (2pt);
\filldraw[gray] (3.9,-0.7) circle (2pt);
\filldraw[gray] (2.9,-0.7) circle (2pt);
\filldraw[gray] (2.9,1.7) circle (2pt);
\filldraw[black] (4.1,0) circle (2pt);
\filldraw[gray] (2.4,1.2) circle (2pt);
\filldraw[gray] (3.4,1.2) circle (2pt);
\filldraw[black] (4.6,0.7) circle (2pt);
\filldraw[black] (4.6,-0.7) circle (2pt);
\filldraw[black] (5.6,0.7) circle (2pt);
\filldraw[black] (5.6,-0.7) circle (2pt);
\filldraw[black] (6.1,0) circle (2pt)node[anchor=west] {$v$};
\end{tikzpicture}
\caption{A nontrivial cactus. The three centers (the hollow vertices) are all contained in a single central cycle (the double-lined cycle). Each of its articulation points $u'$, $v'$, and $w'$ separates some vertices from the central cycle, and one most distant from the central cycle is selected and labelled $u$, $v$, and $w$, respectively. The radius is realized by a subgraph $H_r$ (the thick dark subgraph), the smallest connected nontrivial cactus subgraph containing these vertices and the central cycle. Its BC-tree is starlike with its root corresponding to the central cycle.}\label{fig:ntc_r}
\end{figure}

\begin{thm}\label{thm:ntc_r}
Let $G$ be a nontrivial cactus on $n$ vertices with $k$ cycles. Then
\[R - r \geq (k-1)\left(\sqrt 2 - 1\right) + \frac{1}{2}\]
with equality when $G$ is an even cycle. 
\end{thm}

\begin{proof}
We follow the general scheme of Theorem~\ref{thm:ntc_d}, except for the definition of the subgraph $H=H_r$ from Lemma~\ref{lemma:realize_rd}, the dark subgraph in Figure~\ref{fig:ntc_r}; its BC-tree is starlike with the root corresponding to $B$. 

Lemma~\ref{lemma:rd_vs_n} and Corollary~\ref{cor:chemical_ntc_osv} now combine to give the desired bound, with equality only for an even cycle.

We may now add additional cycles to $H$ to fill out $G$, which increases $R$ by more than $\frac{1}{2} $ according to Lemma~\ref{lemma:cycle} and the right-hand side by only $\sqrt 2 -1$.
\end{proof}

\subsection{Cacti}
Finally, we engage with a generic cactus to complete the proof of Corollary~\ref{cor:d} and improve Yang and Lu's bounds of Theorem~\ref{thm:d}. 

\begin{thm}\label{thm:cactus_d} 
Let $G$ be a cactus with $k>0$ cycles ($t$ of which are 3-cycles) and $b>0$ bridges. Then:
\[R - d \geq -\frac{b}{2} - (k-1)\left(2 - \sqrt 2 \right) - 3 + \frac{2}{\sqrt 3} + 2\sqrt\frac{2}{3}.\]
and
\[R - \frac{d}{2} \geq \frac{n-b}{4} - (k-1)\left(\frac{7}{4} - \sqrt 2 \right) - 3 + \frac{2}{\sqrt 3} + 2\sqrt\frac{2}{3}.\]
with equality when $G$ has BC-tree a path, has two pendants attached to the end cycles (i.e., they correspond to leaves in the BC-tree), and is longitudinally symmetric.
\end{thm}

Again, we can derive a bound for $\frac{R}{d}$ for a cactus whose BC-tree is a path to get the bound below, and it appears to extend to the general case (see Conjecture~\ref{conj:d_new}), but we have some difficulty applying the inductive step:
\[\frac{R}{d} \geq \frac{ n - (k-1)\left(3 - 2\sqrt 2 \right) - 6 + \frac{4}{\sqrt 3} + 4\sqrt\frac{2}{3}}{n+k+b-1}.\]

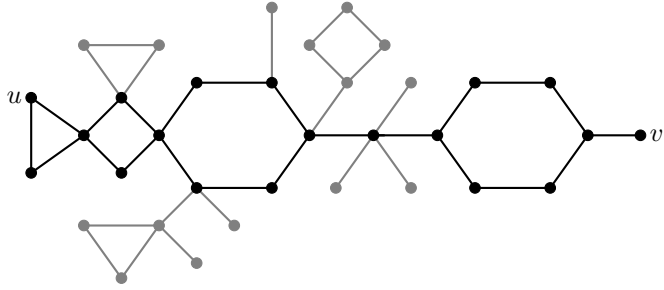
\begin{figure}
\begin{tikzpicture}
\draw[gray, thick] (0.4,-1.2) -- (-0.6,-1.2);
\draw[gray, thick] (0.4,-1.2) -- (-0.1,-1.9);
\draw[gray, thick] (-0.6,-1.2) -- (-0.1,-1.9);

\draw[gray, thick] (-0.1,0.5) -- (0.4,1.2);
\draw[gray, thick] (-0.1,0.5) -- (-0.6,1.2);
\draw[gray, thick] (0.4,1.2) -- (-0.6,1.2);

\draw[black, thick] (-0.6,0) -- (-1.3,0.5);
\draw[black, thick] (-0.6,0) -- (-1.3,-0.5);
\draw[black, thick] (-1.3,0.5) -- (-1.3,-0.5);

\draw[black, thick] (-0.1,0.5) -- (0.4,0);
\draw[black, thick] (-0.1,-0.5) -- (0.4,0);
\draw[black, thick] (-0.1,0.5) -- (-0.6,0);
\draw[black, thick] (-0.1,-0.5) -- (-0.6,0);

\draw[gray, thick] (0.9,-0.7) -- (0.4,-1.2);
\draw[gray, thick] (0.9,-0.7) -- (1.4,-1.2);
\draw[gray, thick] (0.9,-1.7) -- (0.4,-1.2);

\draw[black, thick] (0.4,0) -- (0.9,0.7);
\draw[black, thick] (0.4,0) -- (0.9,-0.7);
\draw[black, thick] (0.9,0.7) -- (1.9,0.7);
\draw[black, thick] (0.9,-0.7) -- (1.9,-0.7);
\draw[black, thick] (1.9,0.7) -- (2.4,0);
\draw[black, thick] (1.9,-0.7) -- (2.4,0);

\draw[gray, thick] (1.9,0.7) -- (1.9,1.7);

\draw[gray, thick] (2.4,0) -- (2.9,0.7);
\draw[black, thick] (2.4,0) -- (3.4,0);

\draw[gray, thick] (2.4,1.2) -- (2.9,0.7);
\draw[gray, thick] (3.4,1.2) -- (2.9,0.7);
\draw[gray, thick] (2.4,1.2) -- (2.9,1.7);
\draw[gray, thick] (3.4,1.2) -- (2.9,1.7);

\draw[gray, thick] (3.25,0) -- (3.75, 0.7);
\draw[black, thick] (3.25,0) -- (4.1, 0);

\draw[gray, thick] (3.25,0) -- (3.75, -0.7);
\draw[gray, thick] (3.25,0) -- (2.75, -0.7);

\draw[black, thick] (4.1,0) -- (4.6, 0.7);
\draw[black, thick] (4.1,0) -- (4.6, -0.7);
\draw[black, thick] (5.6,0.7) -- (4.6, 0.7);
\draw[black, thick] (5.6,-0.7) -- (4.6, -0.7);
\draw[black, thick] (5.6,0.7) -- (6.1, 0);
\draw[black, thick] (5.6,-0.7) -- (6.1, 0);

\draw[black, thick] (6.8,0) -- (6.1, 0);

\filldraw[gray] (-0.1,-1.9) circle (2pt);
\filldraw[gray] (-0.6,-1.2) circle (2pt);
\filldraw[gray] (-0.6,1.2) circle (2pt);
\filldraw[gray] (0.4,1.2) circle (2pt);
\filldraw[black] (-1.3,0.5) circle (2pt)node[anchor=east] {$u$};
\filldraw[black] (-1.3,-0.5) circle (2pt);
\filldraw[black] (-0.1,0.5) circle (2pt);
\filldraw[black] (-0.1,-0.5) circle (2pt);
\filldraw[black] (-0.6,0) circle (2pt);
\filldraw[gray] (0.4,-1.2) circle (2pt);
\filldraw[gray] (1.4,-1.2) circle (2pt);
\filldraw[gray] (0.9,-1.7) circle (2pt);
\filldraw[black] (0.4,0) circle (2pt);
\filldraw[black] (0.9,0.7) circle (2pt);
\filldraw[black] (0.9,-0.7) circle (2pt);
\filldraw[black] (1.9,0.7) circle (2pt);
\filldraw[gray] (1.9,1.7) circle (2pt);
\filldraw[black] (1.9,-0.7) circle (2pt);
\filldraw[black] (2.4,0) circle (2pt);
\filldraw[gray] (2.9,0.7) circle (2pt);
\filldraw[black] (3.25,0) circle (2pt);
\filldraw[gray] (3.75,0.7) circle (2pt);
\filldraw[gray] (3.75,-0.7) circle (2pt);
\filldraw[gray] (2.75,-0.7) circle (2pt);
\filldraw[gray] (2.9,1.7) circle (2pt);
\filldraw[black] (4.1,0) circle (2pt);
\filldraw[gray] (2.4,1.2) circle (2pt);
\filldraw[gray] (3.4,1.2) circle (2pt);
\filldraw[black] (4.6,0.7) circle (2pt);
\filldraw[black] (4.6,-0.7) circle (2pt);
\filldraw[black] (5.6,0.7) circle (2pt);
\filldraw[black] (5.6,-0.7) circle (2pt);
\filldraw[black] (6.1,0) circle (2pt);
\filldraw[black] (6.8,0) circle (2pt)node[anchor=west] {$v$};
\end{tikzpicture}
\caption{A cactus, a graph with cycles and bridges. The diameter is realized by the smallest connected subgraph $H_d$ containing two most distant vertices $u$ and $v$.}\label{fig:cactus_d}
\end{figure}

\begin{proof}
As we did in the previous theorems, we will start with a subgraph $H=H_d$ as in Lemma~\ref{lemma:realize_rd} or the dark subgraph in Figure~\ref{fig:cactus_d}, verify the bounds, and then add additional cycles and bridges.

Note that $H_d$ has BC-tree a path. By Theorem~\ref{thm:star_cactus_osv} and Lemma~\ref{lemma:rd_vs_n}:
\[R - d \geq -\frac{b}{2} - (k-1)\left(2 - \sqrt 2 \right) - 3 + \frac{2}{\sqrt 3} + 2\sqrt\frac{2}{3}.\]
and
\[R - \frac{d}{2} \geq \frac{n-b}{4} - (k-1)\left(\frac{7}{4} - \sqrt 2 \right) - 3 + \frac{2}{\sqrt 3} + 2\sqrt\frac{2}{3}.\]
Equality is possible for the bound on $R$ if there are two bridge which are both pendants, and no cycles have adjacent articulation points; it is possible for the bound on $d$ if additionally the graph is longitudinally symmetric.

Now add cycles and pendants as required to fill out the graph. This increases $R$ by Lemmas~\ref{lemma:leaf} and \ref{lemma:cycle} while leaving $d$ unchanged by assumption; the right-hand side of the bound on $R-d$ is decreasing in both $b$ and $k$, so it holds inductively. For $R-\frac{d}{2}$, note that adding a bridge leaves the right-hand side unaltered, while adding a cycle changes it by $\frac{s}{4} - 2 + \sqrt 2$, while $R$ itself increases by at least $\frac{s}{2} - 1$, so the inequality holds inductively since $s \geq 3$.
\end{proof}

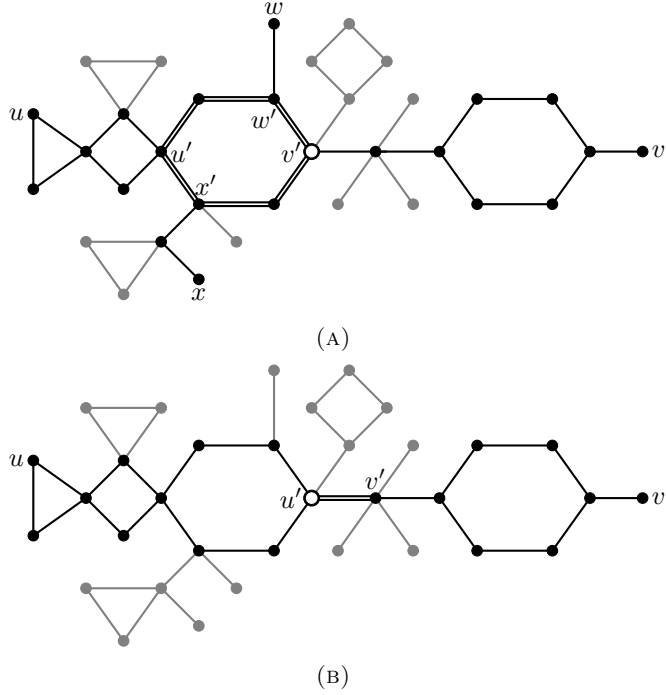
\begin{figure}
\begin{center}
\begin{subfigure}[t]{0.7\textwidth}
\begin{tikzpicture}
\draw[gray, thick] (0.4,-1.2) -- (-0.6,-1.2);
\draw[gray, thick] (0.4,-1.2) -- (-0.1,-1.9);
\draw[gray, thick] (-0.6,-1.2) -- (-0.1,-1.9);

\draw[gray, thick] (-0.1,0.5) -- (0.4,1.2);
\draw[gray, thick] (-0.1,0.5) -- (-0.6,1.2);
\draw[gray, thick] (0.4,1.2) -- (-0.6,1.2);

\draw[black, thick] (-0.6,0) -- (-1.3,0.5);
\draw[black, thick] (-0.6,0) -- (-1.3,-0.5);
\draw[black, thick] (-1.3,0.5) -- (-1.3,-0.5);

\draw[black, thick] (-0.1,0.5) -- (0.4,0);
\draw[black, thick] (-0.1,-0.5) -- (0.4,0);
\draw[black, thick] (-0.1,0.5) -- (-0.6,0);
\draw[black, thick] (-0.1,-0.5) -- (-0.6,0);

\draw[black, thick] (0.9,-0.7) -- (0.4,-1.2);
\draw[gray, thick] (0.9,-0.7) -- (1.4,-1.2);
\draw[black, thick] (0.9,-1.7) -- (0.4,-1.2);

\draw[black, double, thick] (0.4,0) -- (0.9,0.7);
\draw[black, double, thick] (0.4,0) -- (0.9,-0.7);
\draw[black, double, thick] (0.9,0.7) -- (1.9,0.7);
\draw[black, double, thick] (0.9,-0.7) -- (1.9,-0.7);
\draw[black, double, thick] (1.9,0.7) -- (2.4,0);
\draw[black, double, thick] (1.9,-0.7) -- (2.4,0);

\draw[black, thick] (1.9,0.7) -- (1.9,1.7);

\draw[gray, thick] (2.4,0) -- (2.9,0.7);
\draw[black, thick] (2.4,0) -- (3.4,0);

\draw[gray, thick] (2.4,1.2) -- (2.9,0.7);
\draw[gray, thick] (3.4,1.2) -- (2.9,0.7);
\draw[gray, thick] (2.4,1.2) -- (2.9,1.7);
\draw[gray, thick] (3.4,1.2) -- (2.9,1.7);

\draw[gray, thick] (3.25,0) -- (3.75, 0.7);
\draw[black, thick] (3.25,0) -- (4.1, 0);

\draw[gray, thick] (3.25,0) -- (3.75, -0.7);
\draw[gray, thick] (3.25,0) -- (2.75, -0.7);

\draw[black, thick] (4.1,0) -- (4.6, 0.7);
\draw[black, thick] (4.1,0) -- (4.6, -0.7);
\draw[black, thick] (5.6,0.7) -- (4.6, 0.7);
\draw[black, thick] (5.6,-0.7) -- (4.6, -0.7);
\draw[black, thick] (5.6,0.7) -- (6.1, 0);
\draw[black, thick] (5.6,-0.7) -- (6.1, 0);

\draw[black, thick] (6.8,0) -- (6.1, 0);

\filldraw[gray] (-0.1,-1.9) circle (2pt);
\filldraw[gray] (-0.6,-1.2) circle (2pt);
\filldraw[gray] (-0.6,1.2) circle (2pt);
\filldraw[gray] (0.4,1.2) circle (2pt);
\filldraw[black] (-1.3,0.5) circle (2pt)node[anchor=east] {$u$};
\filldraw[black] (-1.3,-0.5) circle (2pt);
\filldraw[black] (-0.1,0.5) circle (2pt);
\filldraw[black] (-0.1,-0.5) circle (2pt);
\filldraw[black] (-0.6,0) circle (2pt);
\filldraw[black] (0.4,-1.2) circle (2pt);
\filldraw[gray] (1.4,-1.2) circle (2pt);
\filldraw[black] (0.9,-1.7) circle (2pt)node[anchor=north] {$x$};
\filldraw[black] (0.4,0) circle (2pt)node[anchor=west] {$u'$};
\filldraw[black] (0.9,0.7) circle (2pt);
\filldraw[black] (0.9,-0.7) circle (2pt)node[anchor=south] {$\phantom{x}x'$};
\filldraw[black] (1.9,0.7) circle (2pt)node[anchor=north] {$w'\phantom{w}$};
\filldraw[black] (1.9,1.7) circle (2pt)node[anchor=south] {$w$};
\filldraw[black] (1.9,-0.7) circle (2pt);
\filldraw[black] (2.4,0) circle (3pt)node[anchor=east] {$v'$};
\filldraw[white](2.4,0) circle(2pt);
\filldraw[gray] (2.9,0.7) circle (2pt);
\filldraw[black] (3.25,0) circle (2pt);
\filldraw[gray] (3.75,0.7) circle (2pt);
\filldraw[gray] (3.75,-0.7) circle (2pt);
\filldraw[gray] (2.75,-0.7) circle (2pt);
\filldraw[gray] (2.9,1.7) circle (2pt);
\filldraw[black] (4.1,0) circle (2pt);
\filldraw[gray] (2.4,1.2) circle (2pt);
\filldraw[gray] (3.4,1.2) circle (2pt);
\filldraw[black] (4.6,0.7) circle (2pt);
\filldraw[black] (4.6,-0.7) circle (2pt);
\filldraw[black] (5.6,0.7) circle (2pt);
\filldraw[black] (5.6,-0.7) circle (2pt);
\filldraw[black] (6.1,0) circle (2pt);
\filldraw[black] (6.8,0) circle (2pt)node[anchor=west] {$v$};
\end{tikzpicture}
\subcaption{}\label{fig:cactus_r1}
\end{subfigure}
\end{center}
\begin{subfigure}[t]{0.7\textwidth}
\begin{center}
\begin{tikzpicture}
\draw[gray, thick] (0.4,-1.2) -- (-0.6,-1.2);
\draw[gray, thick] (0.4,-1.2) -- (-0.1,-1.9);
\draw[gray, thick] (-0.6,-1.2) -- (-0.1,-1.9);

\draw[gray, thick] (-0.1,0.5) -- (0.4,1.2);
\draw[gray, thick] (-0.1,0.5) -- (-0.6,1.2);
\draw[gray, thick] (0.4,1.2) -- (-0.6,1.2);

\draw[black, thick] (-0.6,0) -- (-1.3,0.5);
\draw[black, thick] (-0.6,0) -- (-1.3,-0.5);
\draw[black, thick] (-1.3,0.5) -- (-1.3,-0.5);

\draw[black, thick] (-0.1,0.5) -- (0.4,0);
\draw[black, thick] (-0.1,-0.5) -- (0.4,0);
\draw[black, thick] (-0.1,0.5) -- (-0.6,0);
\draw[black, thick] (-0.1,-0.5) -- (-0.6,0);

\draw[gray, thick] (0.9,-0.7) -- (0.4,-1.2);
\draw[gray, thick] (0.9,-0.7) -- (1.4,-1.2);
\draw[gray, thick] (0.9,-1.7) -- (0.4,-1.2);

\draw[black, thick] (0.4,0) -- (0.9,0.7);
\draw[black, thick] (0.4,0) -- (0.9,-0.7);
\draw[black, thick] (0.9,0.7) -- (1.9,0.7);
\draw[black, thick] (0.9,-0.7) -- (1.9,-0.7);
\draw[black, thick] (1.9,0.7) -- (2.4,0);
\draw[black, thick] (1.9,-0.7) -- (2.4,0);

\draw[gray, thick] (1.9,0.7) -- (1.9,1.7);

\draw[gray, thick] (2.4,0) -- (2.9,0.7);
\draw[black, double, thick] (2.4,0) -- (3.25,0);

\draw[gray, thick] (2.4,1.2) -- (2.9,0.7);
\draw[gray, thick] (3.4,1.2) -- (2.9,0.7);
\draw[gray, thick] (2.4,1.2) -- (2.9,1.7);
\draw[gray, thick] (3.4,1.2) -- (2.9,1.7);

\draw[gray, thick] (3.25,0) -- (3.75, 0.7);
\draw[black, thick] (3.25,0) -- (4.1, 0);

\draw[gray, thick] (3.25,0) -- (3.75, -0.7);
\draw[gray, thick] (3.25,0) -- (2.75, -0.7);

\draw[black, thick] (4.1,0) -- (4.6, 0.7);
\draw[black, thick] (4.1,0) -- (4.6, -0.7);
\draw[black, thick] (5.6,0.7) -- (4.6, 0.7);
\draw[black, thick] (5.6,-0.7) -- (4.6, -0.7);
\draw[black, thick] (5.6,0.7) -- (6.1, 0);
\draw[black, thick] (5.6,-0.7) -- (6.1, 0);

\draw[black, thick] (6.8,0) -- (6.1, 0);

\filldraw[gray] (-0.1,-1.9) circle (2pt);
\filldraw[gray] (-0.6,-1.2) circle (2pt);
\filldraw[gray] (-0.6,1.2) circle (2pt);
\filldraw[gray] (0.4,1.2) circle (2pt);
\filldraw[black] (-1.3,0.5) circle (2pt)node[anchor=east] {$u$};
\filldraw[black] (-1.3,-0.5) circle (2pt);
\filldraw[black] (-0.1,0.5) circle (2pt);
\filldraw[black] (-0.1,-0.5) circle (2pt);
\filldraw[black] (-0.6,0) circle (2pt);
\filldraw[gray] (0.4,-1.2) circle (2pt);
\filldraw[gray] (1.4,-1.2) circle (2pt);
\filldraw[gray] (0.9,-1.7) circle (2pt);
\filldraw[black] (0.4,0) circle (2pt);
\filldraw[black] (0.9,0.7) circle (2pt);
\filldraw[black] (0.9,-0.7) circle (2pt);
\filldraw[black] (1.9,0.7) circle (2pt);
\filldraw[gray] (1.9,1.7) circle (2pt);
\filldraw[black] (1.9,-0.7) circle (2pt);
\filldraw[black] (2.4,0) circle (3pt)node[anchor=east] {$u'$};
\filldraw[white](2.4,0) circle(2pt);
\filldraw[gray] (2.9,0.7) circle (2pt);
\filldraw[black] (3.25,0) circle (2pt)node[anchor=south] {$v'$};
\filldraw[gray] (3.75,0.7) circle (2pt);
\filldraw[gray] (3.75,-0.7) circle (2pt);
\filldraw[gray] (2.75,-0.7) circle (2pt);
\filldraw[gray] (2.9,1.7) circle (2pt);
\filldraw[black] (4.1,0) circle (2pt);
\filldraw[gray] (2.4,1.2) circle (2pt);
\filldraw[gray] (3.4,1.2) circle (2pt);
\filldraw[black] (4.6,0.7) circle (2pt);
\filldraw[black] (4.6,-0.7) circle (2pt);
\filldraw[black] (5.6,0.7) circle (2pt);
\filldraw[black] (5.6,-0.7) circle (2pt);
\filldraw[black] (6.1,0) circle (2pt);
\filldraw[black] (6.8,0) circle (2pt)node[anchor=west] {$v$};
\end{tikzpicture}
\subcaption{}\label{fig:cactus_r2}
\end{center}
\end{subfigure}
\caption{A cactus with one center an articulation point (hollow vertex) which is part of three different blocks. Two of these blocks contain vertices $u$ and $v$ of distance $r$ from the center, so we may have two options for the central block, (A) the cycle with four articulation points and BC-tree starlike with four arms, or (B) the bridge to the right with two articulation points and BC-tree a path as in Figure~\ref{fig:cactus_d}. In either case, the dark black subgraph $H_r$ realizes the radius of $G$.}\label{fig:cactus_r}
\end{figure}

To study the radius, we follow the pattern of Theorem~\ref{thm:ntc_r} with the added complication of possible bridges. This bound is stronger than Conjecture~\ref{conj:r} unless $k=0$, and it allows us to complete the proof of Corollary~\ref{cor:r} since $-\frac{5}{2} + \frac{2}{\sqrt 3} + 2 \sqrt\frac{2}{3} \approx 0.2$.

\begin{thm}\label{thm:cactus_r}
Let $G$ be a cactus with $k>0$ cycles and $b>0$ bridges. Then
\[R - r > (k-1)\left(\sqrt2-1\right).\]
In fact, if a central block has $m\geq 2$ articulation points, then
\[R-r \geq (k-1)\left(\sqrt2-1\right)+m\left(\frac{1}{\sqrt 3} + \sqrt\frac{2}{3}-1\right) - \frac{1}{2}.\]
\end{thm}

\begin{proof} 
We will follow Theorem~\ref{thm:ntc_r}, defining a subgraph $H_r$ as in Lemma~\ref{lemma:realize_rd} and Figure~\ref{fig:cactus_r} with the desired radius, verifying the bounds, and then adding additional cycles and bridges. 


If $k=0$, the bound holds by Theorem~\ref{thm:tree_rd}; if $m=0$, by Theorem~\ref{thm:ntc_r}. Note $m$ is never 1 because then the BC-tree of $H_r$ would be be a path with the central block a leaf. If this central block had radius $r$, then $H_r$ would consist of only that block and $m$ would actually be 0; if the central block had smaller radius, then some adjacent vertex in the next block would also be a center and that block would therefore be the central block. Assuming $m \geq 2$ and $k > 0$, Theorem~\ref{thm:star_cactus_osv} and Lemma~\ref{lemma:rd_vs_n} imply:
\[(R-r)(H_r) \geq (k-1)\left(\sqrt2-1\right)+m\left(\frac{1}{\sqrt 3} + \sqrt\frac{2}{3}-1\right) - \frac{1}{2} > 0.\]

We then fill out $R$ by adding extra and cycles (which by Lemma~\ref{lemma:cycle} increases the left by at least $\frac{1}{2}$ and the right by exactly $\sqrt 2 - 1$) and pendants (which by Lemma~\ref{lemma:leaf} increases the left - in particular, it increases it by $\sqrt 2-1$ if the pendant is added at a degree two vertex; on the other hand, it leaves the right unchanged unless it is a pendant attached to a degree 2 vertex on the central cycle). 
\end{proof}

\section{Future Work}

The ideas at work here may be applicable to generic graphs as well, but it would involve better understanding blocks which are more complicated than cycles or bridges, i.e., understanding the relationship between the Randi\'c index and radius and diameter for graphs without articulation points, as well as understanding how $R$ may change when a generic block is added or removed. 

Additionally, Conjecture~\ref{conj:d_new} suggests some interesting possibilities; when blocks were added or removed above, our primary lemmata dealt with the additive changes to $R$, which enabled studying $R-r$, $R-d$, etc. Extending these results to the multiplicative $\frac{R}{d}$ for trees was possible because the resulting bound was decreasing in $n$; however, the proposed bounds for cacti are also decreasing with the addition of a bridge but not necessarily with the addition of a cycle, which poses a challenge. Perhaps the answer lies in bounding the maximum possible size of a cycle which may be added without affecting diameter, or in the order in which blocks are added.

Next, it is very tempting to look for nicer formulations of the bounds on $R$, perhaps a general one for all cacti which clearly reduces to the desired ones for trees and nontrivial cacti. The challenge seems to be that a graph with both cycles and bridges has $R$ not quite the sum of what would by indicated by examining the cycles or bridges alone; perhaps this indicates the inclusion of a term counting collisions between adjacent blocks of different sorts, i.e., a bridge and a cycle next to one another cause problems.

Finally, the success of theorems like Corollary~\ref{cor:valency} on approximating $R$ via vertex valency suggests that perhaps a separate invariant be defined, one based on vertex valency. How far good would such an invariant be at approximating physicochemical properties of molecules? At studying graph properties like branching?

\bibliographystyle{alpha}
\bibliography{../biblio-MASTER}

\newcommand{\etalchar}[1]{$^{#1}$}
\begin{thebibliography}{GDGdJOP08}

\bibitem[ADLP98]{araujo1998connectivity}
O~Araujo and JA~De~La~Pe{\~{n}}a.
\newblock The connectivity index of a weighted graph.
\newblock {\em Linear Algebra Appl.}, 283(1-3):171--177, 1998.

\bibitem[AH07]{aouchiche2007conjecture}
Mustapha Aouchiche and Pierre Hansen.
\newblock On a conjecture about the {R}andi{\'{c}} index.
\newblock {\em Discrete Math.}, 307(2):262--265, 2007.

\bibitem[AHZ06]{aouchiche2006variable}
Mustapha Aouchiche, Pierre Hansen, and Maolin Zheng.
\newblock Variable neighborhood search for extremal graphs: 18: {C}onjectures
  and results about the {R}andi{\'{c}} index.
\newblock {\em MATCH Commun. Math. Comput. Chem.}, 56(3):541--550, 2006.

\bibitem[AHZ07]{aouchiche2007variable}
Mustapha Aouchiche, Pierre Hansen, and Maolin Zheng.
\newblock Variable neighborhood search for extremal graphs: 19: {F}urther
  conjectures and results about the {R}andi{\'{c}} index.
\newblock {\em MATCH Commun. Math. Comput. Chem.}, 58(1):83, 2007.

\bibitem[Bal82]{balaban1982discriminating}
Alexandru~T Balaban.
\newblock Highly discriminating distance-based topological index.
\newblock {\em Chem. Phys. Lett.}, 89(5):399--404, 1982.

\bibitem[BE98]{bollobas1998extremal}
B{\'{e}}la Bollob{\'{a}}s and Paul Erd{\H{o}}s.
\newblock Graphs of extremal weights.
\newblock {\em Ars Combin.}, 50:225--233, 1998.

\bibitem[CGHP03]{caporossi2003graphs}
Gilles Caporossi, Ivan Gutman, Pierre Hansen, and Ljiljana Pavlovi{\'{c}}.
\newblock Graphs with maximum connectivity index.
\newblock {\em Comput. Biol. and Chem.}, 27(1):85--90, 2003.

\bibitem[CH00]{caporossi2000variable}
Gilles Caporossi and Pierre Hansen.
\newblock Variable neighborhood search for extremal graphs: 1: {T}he
  {A}uto{G}raphi{X} system.
\newblock {\em Discrete Math.}, 212(1-2):29--44, 2000.

\bibitem[CP{\v{S}}12]{cygan2012inequality}
Marek Cygan, Micha{\l{}} Pilipczuk, and Riste {\v{S}}krekovski.
\newblock On the inequality between radius and {R}andi{\'{c}} index for graphs.
\newblock {\em MATCH Commun. Math. Comput. Chem.}, 67(2):451--466, 2012.

\bibitem[DP13]{divnic2013proof}
Tomica~R Divni{\'{c}} and Ljiljana~R Pavlovi{\'{c}}.
\newblock Proof of the first part of the conjecture of {A}ouchiche and {H}ansen
  about the {R}andi{\'{c}} index.
\newblock {\em Discrete Applied Mathematics}, 161(7-8):953--960, 2013.

\bibitem[EW15]{elphick2015bounds}
Clive Elphick and Pawel Wocjan.
\newblock Bounds and power means for the general {R}andi{\'{c}} index.
\newblock {\em arXiv preprint arXiv:1508.07950}, 2015.

\bibitem[Faj88]{fajtlowicz1988graffiti}
Siemion Fajtlowicz.
\newblock On conjectures of {G}raffiti.
\newblock {\em Discrete Math.}, 72(1-3):113--118, 1988.

\bibitem[FMS03]{favaron2003randic}
O~Favaron, M~Mah\'eo, and JF~Sacle.
\newblock The {R}andi{\'{c}} index and other {G}raffiti parameters of graphs.
\newblock {\em MATCH Commun. Math. Comput. Chem.}, 47:7--23, 2003.

\bibitem[GDGdJOP08]{garcia2008some}
Ram{\'{o}}n Garc{\'{i}}a-Domenech, Jorge G{\'{a}}lvez, Jesus~V.
  de~Juli{\'{a}}n-Ortiz, and Lionello Pogliani.
\newblock Some new trends in chemical graph theory.
\newblock {\em Chem. Rev.}, 108(3):1127--1169, 2008.

\bibitem[GPM00]{gutman2000graphs}
I~Gutman, Ljiljana Pavlovi{\'{c}}, and Olga Miljkovi{\'{c}}.
\newblock On graphs with extremal connectivity indices.
\newblock {\em Bulletin ({A}cad{\'e}mie {S}erbe des {S}ciences et des {A}rts.
  {C}lasse des {S}ciences {M}ath{\'e}matiques et {N}aturelles. {S}ciences
  {M}ath{\'e}matiques)}, 25:1--14, 2000.

\bibitem[KH76]{kier2012molecular}
Lemont~B. Kier and Lowell~H. Hall.
\newblock {\em Molecular {C}onnectivity in {C}hemistry and {D}rug {R}esearch}.
\newblock Academic Press, New York, 1976.

\bibitem[KH86]{kier1986molecular}
Lemont~Burwell Kier and Lowell~H. Hall.
\newblock {\em Molecular {C}onnectivity in {S}tructure-{A}ctivity {A}nalysis}.
\newblock Research Studies Press, Letchworth, Hertfordshire, England, 1986.

\bibitem[LG09]{liu2009conjecture}
Bolian Liu and Ivan Gutman.
\newblock On a conjecture on {R}andi{\'{c}} indices.
\newblock {\em MATCH Commun. Math. Comput. Chem.}, 62(1):143--154, 2009.

\bibitem[LLCL11]{liu2011proof}
Jianxi Liu, Meili Liang, Bo~Cheng, and Bolian Liu.
\newblock A proof for a conjecture on the {R}andi{\'{c}} index of graphs with
  diameter.
\newblock {\em Appl. Math. Lett.}, 24(5):752--756, 2011.

\bibitem[LPD{\etalchar{+}}13]{liu2013conjecture}
Bolian Liu, Ljiljana~R Pavlovi{\'{c}}, Tomica~R Divni{\'{c}}, Jianxi Liu, and
  Marina~M Stojanovi{\'{c}}.
\newblock On the conjecture of {A}ouchiche and {H}ansen about the
  {R}andi{\'{c}} index.
\newblock {\em Discrete Math.}, 313(3):225--235, 2013.

\bibitem[LS10a]{li2010relation}
Xueliang Li and Yongtang Shi.
\newblock On a relation between the {R}andi{\'{c}} index and the chromatic
  number.
\newblock {\em Discrete Math.}, 310(17-18):2448--2451, 2010.

\bibitem[LS10b]{li2009randi}
Xueliang Li and Yongtang Shi.
\newblock Randi{\'{c}} index, diameter and the average distance.
\newblock {\em MATCH Commun. Math. Comput. Chem.}, 64(2):425--431, 2010.

\bibitem[LZT06]{lu2006randic}
Mei Lu, Lianzhu Zhang, and Feng Tian.
\newblock On the {R}andi{\'{c}} index of cacti.
\newblock {\em MATCH Commun. Math. Comput. Chem.}, 56(3):551--556, 2006.

\bibitem[Pog00]{pogliani2000molecular}
Lionello Pogliani.
\newblock From molecular connectivity indices to semiempirical connectivity
  terms: {R}ecent trends in graph theoretical descriptors.
\newblock {\em Chem. Rev.}, 100(10):3827--3858, 2000.

\bibitem[Ran75]{randic1975characterization}
Milan Randi{\'{c}}.
\newblock Characterization of molecular branching.
\newblock {\em J. Am. Chem. Soc.}, 97(23):6609--6615, 1975.

\bibitem[SS18]{suil2018sharp}
O~Suil and Yongtang Shi.
\newblock Sharp bounds for the {R}andi{\'{c}} index of graphs with given
  minimum and maximum degree.
\newblock {\em Discret. Appl. Math.}, 247:111--115, 2018.

\bibitem[TC08]{todeschini2008handbook}
Roberto Todeschini and Viviana Consonni.
\newblock {\em Handbook of molecular descriptors}, volume~11.
\newblock John Wiley \& Sons, 2008.

\bibitem[YL09]{you2009conjecture}
Zhifu You and Bolian Liu.
\newblock On a conjecture of the {R}andi{\'{c}} index.
\newblock {\em Discret. Appl. Math.}, 157(8):1766--1772, 2009.

\bibitem[YL11]{yang2011randic}
Yiting Yang and Linyuan Lu.
\newblock The {R}andi{\'{c}} index and the diameter of graphs.
\newblock {\em Discrete Math.}, 311(14):1333--1343, 2011.

\bibitem[ZL10]{zhang2010conjecture}
Meng Zhang and Bolian Liu.
\newblock On a conjecture about the {R}andi{\'{c}} index and diameter.
\newblock {\em MATCH Commun. Math. Comput. Chem.}, 64(2):433--442, 2010.

\end{thebibliography}

\end{document}